\documentclass[11pt]{article}
\usepackage[ansinew]{inputenc}
\usepackage{graphicx}
\usepackage{color}
\usepackage{enumerate,latexsym}
\usepackage{latexsym}
\usepackage{amsmath,amssymb}
\usepackage{graphicx}
\usepackage[margin=1.4in]{geometry}
\newfont{\bb}{msbm10 at 11pt}
\newfont{\bbsmall}{msbm8 at 8pt}

\def\rth{\mathbb{R}^3}
\def\R{\mathbb{R}}
\def\g{\gamma}

\def\B{\mathbb{B}}
\def\N{\mathbb{N}}
\def\Z{\mathbb{Z}}

\def\esf{\mathbb{S}}
\def\T{\mathbb{T}^3}

\newcommand{\nc}{\newcommand}
\newcommand{\ben}{\begin{enumerate}}
\newcommand{\bit}{\begin{itemize}}
\newcommand{\een}{\end{enumerate}}
\newcommand{\eit}{\end{itemize}}
\newcommand{\wh}{\widehat}

\newcommand{\cC}{{\mathcal C}}
\newcommand{\cT}{\mathcal{T}}

\newcommand{\cR}{{\cal R}}
\newcommand{\cL}{{\cal L}}
\newcommand{\cH}{{\cal H}}

\newcommand{\wt}{\widetilde}

\newcommand{\bp}{\begin{proof}}
\newcommand{\ep}{\end{proof}}

\def\a{{\alpha}}
\def\ov{\overline}
\def\g{{\gamma}}
\def\G{{\Gamma}}
\def\l{{\lambda}}
\def\L{{\Lambda}}
\def\de{{\delta}}
\def\T{\mathbb{T}}

\def\ve{{\varepsilon}}

\def\centerbmp#1#2#3{\vskip#2\relax\centerline{\hbox to#1{\special
    {bmp:#3 x=#1, y=#2}\hfil}}}

\newtheorem{theorem}{Theorem}[section]
\newtheorem{lemma}[theorem]{Lemma}
\newtheorem{proposition}[theorem]{Proposition}

\newtheorem{remark}[theorem]{Remark}

\newtheorem{corollary}[theorem]{Corollary}
\newtheorem{definition}[theorem]{Definition}
\newtheorem{conjecture}[theorem]{Conjecture}
\newtheorem{assertion}[theorem]{Assertion}

\newtheorem{claim}[theorem]{Claim}

\newcommand{\ed}{\end{document}}
\nc{\bl}{\begin{lemma} }

\nc{\el}{\end{lemma} }

\nc{\bt}{\begin{theorem} }

\nc{\et}{\end{theorem} }
\newcommand{\rc}{ \renewcommand }

\rc{\v}{    \overset{\longrightarrow} }

\newenvironment{proof}{\smallskip\noindent{\it Proof.}\hskip \labelsep}
{\hfill\penalty10000\raisebox{-.09em}{$\Box$}\par\medskip}

\newcommand{\rrr}{\textcolor{rr}}
\definecolor{rr}{rgb}{.8,0,.3}

\begin{document}

\begin{title}
{Triply periodic constant mean curvature surfaces }
\end{title}
\begin{author}
{William H. Meeks III\thanks{The first author was supported in part by NSF Grant DMS-1309236.
Any opinions, findings, and conclusions or recommendations
   expressed in this publication are those of the authors and do not
   necessarily reflect the views of the NSF.} \and Giuseppe
   Tinaglia\thanks{ The second author was partially supported by EPSRC grant no. EP/M024512/1}
 }
\end{author}
\maketitle
\begin{abstract}
Given a closed flat 3-torus $N$, for each $H>0$ and each non-negative integer $g$,
we obtain area estimates for  closed surfaces with genus $g$ and  constant mean curvature
$H$ embedded in $N$.
This result contrasts with
the theorem of Traizet~\cite{tra5}, who proved  that every flat
3-torus  admits for every positive integer $g$ with $g\neq 2$,  connected closed embedded
minimal surfaces of genus $g$ with arbitrarily large area.
\par
\vspace{.1cm} \noindent{\it Mathematics Subject Classification:}
Primary 53A10, Secondary 49Q05, 53C42.

\vspace{.1cm} \noindent{\it Key words and phrases:} Minimal surface,
constant mean curvature,  $H$-lamination,
injectivity radius,  curvature estimates, area estimates.
\end{abstract}


\section{Introduction}

In~\cite{tra5} Traizet proved that every flat
3-torus  admits for every positive integer $g$ with $g\neq 2$,  connected closed embedded
minimal surfaces of genus $g$  with arbitrarily large area.
In contrast to Traizet's result
we prove the following area estimates for closed embedded surfaces of constant positive mean curvature
in a flat 3-torus.

\begin{theorem} \label{area} Given
$a,b,d,I_0\in (0,\infty)$ with $a\leq b$ and $g\in \N\cup\{0\}$, there exists
$A(g,a,b,d,I_0)>0$ such that the following hold. Let
$N$ be a flat 3-torus with an upper bound $d$ on its
diameter  and a lower bound $I_0$ for its injectivity radius and let
$M$ be a  possibly disconnected, closed
surface embedded in  $N$ of genus $g$ with constant
mean curvature  $H\in[a,b]$. Then:
     $$\mathrm{Area}(M)\leq A(g,a,b,d,I_0).$$
\end{theorem}


Interest in results like the above  area estimates  arises in part  from
the fact that triply periodic surfaces of constant mean curvature in $\rth$, which are always the lifts of
surfaces of
constant mean curvature in a flat 3-torus, occur in nature.
For example, approximations to these
special surfaces appear
in studies
of Fermi surfaces (equipotential surfaces) in solid state
physics and in the geometry  of liquid crystals.
They are also found in material sciences where they  closely approximate surface interfaces
in certain inhomogeneous mixtures of two different compounds and
as the boundary of the microscopic  calcium deposit patterns  in
sea urchin shells. The geometry of
triply-periodic constant mean curvature surfaces
arising in nature
have profound consequences for the physical properties of the materials in which they occur,
which is one of the  reasons why they are studied in great detail by physical scientists.
We refer the interested reader to the science articles \cite{aht2,mac1,scr1} for further
  readings along these lines.

For the purpose of exposition, in this paper we will call an orientable immersed surface
$M$ in an oriented Riemannian 3-manifold an {\it $H$-surface} if $M$  is
{\it embedded} and it
has {\it non-zero constant mean curvature $H$}, where we will assume
that $H$ is  positive by appropriately orienting $M$.

A consequence of Theorem~\ref{area} and its proof is that for $H>0$ fixed,  the moduli
space of  closed connected $H$-surfaces  of fixed  genus  in a flat 3-torus  has a
natural compactification as a compact real semi-analytic variety
in much the same way that one can compactify the
moduli space of closed connected Riemann surfaces of fixed genus.
In this regard it is worthwhile to
compare our results with  the area estimates
of Choi and Wang~\cite{cw1} for closed embedded minimal surfaces of fixed genus in
the 3-sphere $\langle \esf^3, h\rangle$ with a metric $h$ of positive
Ricci curvature, and with the  result by Choi and Schoen~\cite{cs1} that
the moduli space of fixed genus closed minimal surfaces embedded in $\langle \esf^3, h\rangle$
has the structure of  a compact real analytic variety.

\begin{theorem} \label{seq-compact}
Let $N$ be a flat 3-torus. Given
$a,b\in (0,\infty)$ with $a\leq b$ and $g\in \N\cup\{0\}$, let $\{M_n\}_{n\in \N}$ be a
sequence of closed $H_n$-surfaces  in
$N$, $H_n\in[a,b]$, of genus $g$. Then there exist a
subsequence of  $\{M_n\}_{n\in \N}$  and a non-empty,
possibly disconnected, strongly Alexandrov
embedded\footnote{A closed immersed surface $f\colon \Sigma\to N$ of
positive mean curvature in  $N$
is called {\em strongly Alexandrov embedded} if $f$ extends  on the
mean convex side of its image to an immersion of a compact 3-manifold $W$
with $\Sigma=\partial W$, where the
extended immersion
is injective on the interior of $W$. In particular, such a
strongly Alexandrov embedded $\Sigma$ can
be approximated by  embedded surfaces on its mean convex side.}
surface $M_{\infty}$ of constant  mean curvature  $H_{\infty}\in [a,b]$ and of
 genus at most $g$, such that the following holds.
The surfaces in this  subsequence converge smoothly with multiplicity one to $M_{\infty}$ away
from a finite set of points $\Delta$ (of singular points of convergence)
contained in the self-intersection set
of $M_{\infty}$.\end{theorem}

Suppose that $\{M_n\}_{n\in \N}$ is the subsequence that converges to the
limit surface $M_\infty$ in the above theorem with singular
set of convergence $\Delta$. The {\em convergence with multiplicity one} property
 implies that the areas of $M_n$ converge to the area
of the limit surface.
In a sense that is made clear  in Proposition~\ref{singnogenus},
nearby every  $q\in \Delta$ (on the scale of the injectivity radius), the surface $M_n$ contains
domains that look like a scaled catenoid for $n$ large.  It also follows from  Proposition~\ref{singnogenus} that
$M_\infty $ has  area density 2 at each  $q \in \Delta$.
We conjecture
 that for every $q\in \Delta$ there exists $\ve>0$ small such that
$B_N(q,\ve)\cap M_n$ is an annulus  for $n$ sufficiently  large, where $B_N(q,\ve)$ denotes the open ball of $N$ centered at $q$
of radius $\ve$ (see Conjecture~\ref{genus-0}).

The manuscript is organized as follows. In Section~2 we list a few results from other papers that will be
needed in the proof of Theorem~\ref{area}. In Section 3 we
describe some separation properties of certain $H$-surfaces
in a complete flat 3-manifold. In Sections 4, 5, 6, 7 and 8 we prove a version of Theorem~\ref{area},
where the area bound depends in a non-specified manner on the ambient 3-torus $N$. Arguing
by contradiction, we suppose there exists a sequence
$\{M_n\}_{n\in \N}$ of closed  $H_n$-surfaces, $H_n\in [a,b]$, of genus $g$ with
arbitrary large area in $N$. Then, in Section~4 we prove that, after replacing by a subsequence,
the injectivity radii of the surfaces $M_n$ must be going to zero. In Section~5 we analyze the local
geometry of $M_n$ nearby a point where the injectivity radius is becoming
arbitrarily small, such points are called {\em singular points}.
In Section~6 we analyze a global consequence for the geometry of $M_n$ if there
exist singular points. In Section~7 we use this analysis to prove that the number
of singular points is finite. In Section 8 we apply the fact that the number of
singular points is finite to obtain a contradiction
to the assumed  hypothesis that the area is becoming arbitrarily large.
In Section~9 we prove the compactness theorem, namely
Theorem~\ref{seq-compact}, as well as a stronger characterization of the
geometry of $M_n$ in a small neighborhood
of singular point, when $n$ is sufficiently large.
In Section~10 we analyze the constant $A(g,a,b,d,I_0)$
that gives the area bound in Theorem~\ref{area} and complete its proof.
In Section~11  we discuss the dependence of $A(g,a,b,d,I_0)$
on the genus $g$, on the lower and upper bounds $a$ and $b$ for the constant mean curvature $H$
and on the diameter $d$ of the flat 3-torus, under the assumption that $I_0$ is fixed.
In Section~12, we  promote several conjectures related to our main theorems.

\vspace{.2cm}

\noindent  {\sc Acknowledgements:} The authors would
like to thank Karsten  Gro\ss e-Brauckmann
for giving us the  images of triply periodic
$H$-surfaces that appear in Figure~\ref{star}.

\section{Preliminaries}  \label{sec:pre}
We now recall several key notions and
theorems from \cite{mt15} that we will apply in the next section.
\begin{definition} \label{definitions} {\rm Let $M$ be an $H$-surface,
possibly with boundary, in a complete oriented
Riemannian 3-manifold $N$. \ben \item $M$  is an {\em $H$-disk}\,
if $M$ is diffeomorphic to a closed
 disk in the complex plane.
\item $|A_{M}|$ denotes the {\em norm of the second fundamental form} of  $M$.
\item The {\em radius} of a  Riemannian $n$-manifold with boundary is the supremum of the
intrinsic distances of points in the manifold to its boundary. \een }
\end{definition} \vspace{.1cm}

The  next two
 results are contained in~\cite{mt15}, see also~\cite{mt8,mt7,mt13,mt14,mt9}.

\begin{theorem}[Radius Estimates] \label{rest} There exists an ${\cal R}\geq \pi$
such that any $H$-disk in $\rth$ has radius less than ${\cal R}/{H}$.
\end{theorem}

\begin{theorem}[Curvature Estimates] \label{cest} Given $\delta,\cH>0$, there exists
a  $K(\delta,\cH)\geq \sqrt 2\cH$ such that for any $H$-disk ${\cal D}\subset \rth$ with $H\geq \cH$,
$${\large{\LARGE \sup}_{\large \{p\in {\cal D} \, \mid \,
d_{\cal D}(p,\partial {\cal D})\geq \delta\}} |A_{\cal D}|\leq  K(\delta,\cH)}.$$
\end{theorem}

The following notions of injectivity radius function and injectivity radius are needed in
the statement and the proof of Theorem~\ref{area}.
\begin{definition} \label{Inj} {\em
The injectivity radius $I_M(p)$ at a point $p$ of a complete Riemannian manifold $M$ is
the supremum of the  radii $R$ for which the exponential
map on the open ball of radius $R$ in $T_pN$ is a diffeomorphism. This defines
the {\em injectivity radius function}, $I_M\colon M\to (0,\infty ]$, which is continuous
on $M$ (see e.g., Proposition~88 in Berger~\cite{ber1}). The infimum of
$I_M$ is called the {\em injectivity radius} of $M$.
}
\end{definition}


\begin{corollary} \label{corinj} If $M$ is a
complete  $H$-surface in $\rth$ with positive injectivity radius ${r_0}$, then
 $$\sup_M|A_M|\leq K(r_0,H).$$
Furthermore,  such an $M$ is properly embedded in $\rth$ and it
is the oriented  boundary of a smooth,
possibly disconnected, mean convex closed domain $G_M$ in $\rth$ which
has a 1-sided regular $\ve$-neighborhood for its
boundary for some $\ve>0$.
\end{corollary}

\begin{proof}
The first statement in the  corollary is an immediate
consequence of Theorem~\ref{cest}. We next explain how the second statement follows from the first one.
First note that complete connected $H$-surfaces in $\rth$ of bounded norm of the
second fundamental form are proper; see for instance~\cite{mr8} for this   result.
It follows from~\cite{mr8} that for some $\ve>0$, each component of $M$ has a
regular $\ve$-neighborhood on its mean convex side, where $\ve>0$ only depends
on the bound of the norm of the second
fundamental form of $M$.  In~\cite{ror1} Ros and Rosenberg proved that
given two complete disjoint connected
proper $H$-surfaces in $\rth$, neither one lies on the mean convex side of the other one.
It then  follows from elementary separation properties
that a complete, possibly disconnected, $H$-surface $M$ in $\rth$ of bounded norm of the
second fundamental form is   proper, it is the oriented  boundary of
a possibly disconnected, mean convex  closed  domain $G_M$ in $\rth$ and, for some $\ve>0$, $G_M$
has a 1-sided regular $\ve$-neighborhood for its
boundary $M$.
\end{proof}

For a point $p$ in a Riemannian
manifold $N$, we let $\B_N(p,R)$ denote the open ball of $N$ centered at $p$
of radius $R$ and when $N=\rth$, we let $\B(R)$ be the ball centered
at the origin of radius $R$; we let $\ov{B}_N(p,R)$ denote the
related closed balls.

Next we state a result that is closely related to Corollary~\ref{corinj} and that is
needed in the proof of Theorem~\ref{area}.
By applying the techniques used to prove
Theorem~3.5 in~\cite{mt3}, one obtains the following result.

\begin{proposition}\label{area51}
Given $R>0$, $\alpha>0$ and $\beta>0$, there exists a constant $\omega(R,\alpha,\beta)$
such that the following holds.
Suppose $M\subset \ov{\B}(R)$ is an $H$-surface
with  $\partial M\subset \partial \B(R)$,
$H\geq \alpha$ and $\sup_{M}|A_M|<\beta$ and such that $M$ bounds a mean convex domain in $\ov{\B}(R)$.
Then
\[
\textrm{\rm Area}(M\cap \B( R/2))<\omega(R,\alpha,\beta).
\]
\end{proposition}

%

We next describe the notion of
flux of an $H$-surface in $\rth$, see for instance~\cite{kks1,ku2,smyt1}
for further discussion of this invariant. 

\begin{definition} \label{def:flux} {\em
Let $\gamma$ be a piecewise-smooth 1-cycle in an $H$-surface $M$ in $\rth$. The {\em flux} of
$\gamma$ is $F(\g)=|\int_{\gamma}(H\gamma+\xi)\times \dot{\gamma}|$, where $\xi$
is the unit normal to $M$ along $\gamma$ and $\gamma$ is parameterized
by arc length. The flux only depends on the
homology class of $\g$.

In the case that $H_1(M)=\Z$, we let
$F(M)$ denote the flux of any curve which represents a generator of $H_1(M)$.}
\end{definition}
The next theorem appears  in~\cite{mt15}.

\begin{theorem}\label{cann}
Given $\rho>0$ and $\delta\in (0,1)$
there exists a positive constant $I_0(\rho,\delta)$
  such that if $E$ is
a compact 1-annulus with $F(E)\geq \rho$  or with $F(E)=0$,
then $$\inf_{\{p\in E \; \mid \; d_E(p,\partial E)\geq \delta\}} I_E\geq I_0(\rho,\delta), $$
where $I_E\colon E\to [0,\infty)$ is the injectivity radius function of $E$.
\end{theorem}

Corollary~\ref{corinj} and
Theorem~\ref{cann} imply that annular ends of complete   $H$-surfaces in $\rth$
have representatives which are properly embedded in $\rth$, and so by the classification
results in~\cite{kks1}, we have the following.

\begin{corollary} \label{unduloids}
Annular ends of complete   $H$-surfaces in $\rth$
have representatives  that are asymptotic to the ends of unduloids.
\end{corollary}

\section{Properties of $H$-surfaces in a flat 3-manifold}

In this section we prove some geometric properties of complete $H$-surfaces in a flat 3-manifold.

\begin{theorem} \label{separate} Let $N$ be a complete connected  flat 3-manifold with
universal cover $\Pi\colon \rth \to N$ and let $M$ be a complete  $H$-surface
 in $N$. Then the following holds:
\ben
\item  \label{ait1} If $M$ has positive injectivity
radius,
then it has bounded norm of the second fundamental form, is properly
embedded in $N$ and it is the oriented boundary of a smooth, possibly disconnected,
complete closed subdomain $G_M$ on its mean convex side. The mean convex domain $G_M$
has radius
at most $1/H$, and $M=\partial G_M$ has a 1-sided regular $\ve$-neighborhood in $G_M$ for some $\ve>0$.

\item  \label{ait2} If $M$ has finite topology, then it has positive injectivity
radius. Furthermore, each annular end $E$ of $M$ lifts to an annulus
$\wt{E}\subset \rth$, where $\wt{E}$ is asymptotic to the end of an unduloid.
\een
\end{theorem}
\begin{proof}
We first prove item~\ref{ait1} of the theorem.
Suppose that $M$ has positive injectivity radius. Consider the possibly disconnected
surface  $\Sigma=\Pi^{-1}(M)\subset \rth$,
which also has positive
injectivity radius. Applying Corollary~\ref{corinj} to $\Sigma$, we conclude that
$\Sigma$ is properly embedded in $\rth$ with bounded norm of the second fundamental form
and $\Sigma$ is the boundary of a mean convex closed domain $G_\Sigma$ that has  a
regular $\ve$-neighborhood for its boundary surface $\Sigma$ for some $\ve>0$.
By  elementary theory of covering spaces, it follows that the domain $G_M=\Pi(G_\Sigma)$
satisfies all of the properties in item~1 of the theorem except possibly for the   property
that the radius of $G_{M}$ is at most $1/H$. Arguing by
contradiction suppose that there exists a point $p\in G_{M}$
with $d_N(p,\partial G_{M})=R_0>1/H$, where $d_N$ is the distance function in $N$.
In this case consider the
associated mean convex region $G_{\Sigma}=\Pi^{-1}(G_{M})\subset \rth$.
For any choice $\wt{p}\in \Pi^{-1}(p)\cap G_{\Sigma}$,
$d_{\rth}(\wt{p},\partial G_{\Sigma})=R_0>1/H$, where $d_{\rth}$ is the Euclidean distance.
Note that $\partial \B(\wt{p},R_0)$
intersects $\partial G_{\Sigma}$ at some point $q$ and  $\partial \B(\wt{p},R_0)\subset G_\Sigma$.
Since the mean curvature of $\partial G_{\Sigma}$
is $H$ and the mean curvature of $\partial \B(\wt{p},R_0)$
is $1/R_0 <H$, we obtain a contradiction to the mean curvature
comparison principle applied to the surfaces $\partial \B(\wt{p},R_0)$
and $\partial G_{\Sigma}$ at the point $q$.
This  completes the
proof  of item~\ref{ait1} of Theorem~\ref{area}.

We next prove   item~\ref{ait2}.
Let $E$ be an annular end representative of $M$.
We claim that the
inclusion map $i\colon E\to M\subset N$ lifts through
$\Pi\colon\rth\to N$ to $\wt{i} \colon E\to \rth$,
from which  item~\ref{ait2} follows by applying Corollary~\ref{unduloids}
to $\wt{i}(E)$.
 Clearly each component of $\Pi^{-1}(E)\subset \rth$ must be an annulus $A$ with
 boundary $\Pi(\partial A)= \partial E$.
Otherwise, elementary covering space theory implies that each component of
$\Pi^{-1}(E)\subset \rth$ is a complete  simply connected
$H$-surface with boundary and having infinite radius. Such a component
contains $H$-disks of arbitrarily large radius,
thereby contradicting Theorem~\ref{rest}. Suppose that the
inclusion map $i\colon E\to M\subset N$ does not lift through
$\Pi\colon\rth\to N$ to a continuous map $\wt{i} \colon E\to \rth$. In this case there
exist $p,q\in A$, $p\neq q$ such that $\Pi(p)=\Pi(q)\in i(E)$.  Let
$\sigma\colon\rth\to\rth$ be the covering transformation such that $\sigma (p)=q$.
The map $\sigma$ is an isometry of $\rth$ that leaves invariant the compact set $\partial A$.
This implies that $\sigma$ has a fixed point and therefore cannot be a
nontrivial covering transformation.
This contradiction completes the proof of item~\ref{ait2},
thus finishing the proof of the theorem.
\end{proof}

\section{Area estimates for $H$-surfaces in a flat 3-torus}
In the next six sections, $N$ will denote a flat 3-torus. We begin by proving a weaker
area estimate than the one in Theorem~\ref{area},
namely an area estimate for $M$ that depends on the flat 3-torus but
without specifying what geometric quantities of the 3-torus are important in such
estimates. In Section~\ref{sec:10} we outline how
Theorem~\ref{area} will follow from Theorem~\ref{area2} below.
Notice that the $H$-surface $M$ in the next theorem is assumed
to be {\em connected}, whereas the surface in
Theorem~\ref{area} may be disconnected. Indeed, the next theorem actually holds in the more general
situation where $M$
is allowed to be disconnected. This follows from the fact that the number of components of
$M$ can be bounded in terms of its genus and the geometry of $N$; see Section~\ref{sec:8.1}.

\begin{theorem} \label{area2} Let $N$ be a flat 3-torus. Given $a,b\in (0,\infty)$,
with $a\leq b$, and $g\in\mathbb N\cup \{0\}$, there exists a constant
$A(N, a,b, g)$ such that if  $M$ is a closed connected $H$-surface in $N$ of genus $g$
with $H\in [a,b]$, then
     $$\mathrm{Area}(M)\leq A(N,a,b,g).$$
\end{theorem}
\begin{proof}
Arguing by contradiction, suppose $M_n$ is a sequence
of $H_n$-surfaces, $H_n\in [a,b]$,  in $N$ of genus $g$ such that
$$\mathrm{Area}(M_n)>n .$$
The proof  of Theorem~\ref{area2}  is divided into steps with the final contradiction
appearing at the end of Section~\ref{sec:final-cont}.

Note that by item~\ref{ait1} of Theorem~\ref{separate}, $M_n$ separates $N$ into two
regions and one of them, denoted by $G_{M_n}$, is mean convex.

\begin{claim}\label{injec}
The injectivity radii $I(M_n)$ converge to zero as $n$ goes to infinity.
\end{claim}

\begin{proof}
Arguing by contradiction, suppose there exists $\de>0$ such that after replacing by
a subsequence, $I(M_n)>\de$ for any $n$. Then by
Theorem~\ref{cest}, the set of functions $\{|A_{M_n}|\}_n$ is bounded from above by a
fixed constant independent of $n$.
Since the surfaces $M_n$ are $H_n$-surfaces with $H_n\geq a>0$, then Theorem~3.5 in~\cite{mt3}
implies that
there exist constants $\ve, A_0 >0$ such that the surfaces
have a 1-sided regular $\ve$-neighborhood $\mathcal{N}(n,\ve)\subset G_{M_n}$  and the area
of each $M_n$ is at most $A_0\cdot \mathrm{Volume}(\mathcal{N}(n,\ve))$.   Therefore,
$$\mathrm{Area}(M_n)\leq A_0\, \mathrm{Volume}(\mathcal{N}(n,\ve)) \leq  A_0\,  \mathrm{Volume}(N),$$
which contradicts that the areas of the surfaces $M_n$ are becoming arbitrarily large.
\end{proof}

In light of Claim~\ref{injec} we introduce the following definitions.

   \begin{definition} \label{def:lbsf} {\rm Let $U$ be an open set in $N$. \ben \item
We say that a sequence of surfaces $T_n\subset U$  has
{\em locally bounded norm of the second fundamental form in $U$} if for each
compact ball $B$ in $U$, the norms of the second fundamental forms of the surfaces $T_n\cap B$
are uniformly bounded. \item We say that a sequence of surfaces $T_n\subset U$  has
{\em locally positive injectivity radius in $U$} if for each compact ball
$B$ in $U$ the injectivity radius functions  of the surfaces $T_n$ are bounded
away from 0 in $T_n\cap U$. \een } \end{definition}

 Suppose that the injectivity radius functions $I_n$ of $M_n$ have their minimum values
at points $p_{1,n}\in M_n$.   By Claim~\ref{injec} we can assume
that $I_n(p_{1,n})<1/n$.  After choosing a subsequence and reindexing, we obtain a sequence $M_{1,n}$
such that the points $p_{1,n}\in M_{1,n}$ converge to a point $q_1\in N$.  Suppose the sequence of
surfaces $M_{1,n}$ fails to have locally bounded injectivity radius in $N-\{q_1\}$.
Let $q_2\in N-\{q_1\}$ be a point that is furthest away from $q_1$ and such that,
after passing to a subsequence $M_{2,n}$,  there exists a sequence of points $p_{2,n}\in M_{2,n}$
converging to $q_2$ with $\lim_{n\to\infty}I_n(p_{2,n})=0$. If the sequence of
surfaces $M_{2,n}$ fails to have locally bounded injectivity radius in $N-\{q_1,q_2\}$, then
let $q_3\in N-\{q_1\}\cup\{q_2\}$ be a point in $N$ that is furthest away from $\{q_1, q_2\}$
and such that, after passing to a subsequence,  there exists a sequence of
points $p_{3,n}\in M_{3,n}$ converging to $q_3$ with $\lim_{n\to\infty}I_n(p_{3,n})=0$.

Continuing inductively in this manner and using a diagonal-type argument, we obtain
after reindexing,
a new subsequence $M_n$ (denoted in the same way) and
a countable (possibly finite) non-empty set $\Delta':=\{q_1,q_2, q_3,\dots \}\subset \rrr{\bf  N} $
such that for every $k\in \N$, there is an integer $N(k)$ such that for $n\geq N(k)$, there
exist points $p(n,q_k)\in M_n\cap B_N(q_k,1/n)$ where $I_{M_n}(p(n,q_k))<1/n$.
We let $\Delta$ denote the closure of $\Delta'$ in $N$. It follows
from the construction of $\Delta$ that
the sequence $M_n$ has locally positive injectivity radius in $N-\Delta$.

We call the set $\Delta$ the {\em singular set of convergence} and $q\in\Delta$
a {\em singular point}. Note that, by using Theorem~\ref{cest},
$M_n$ has   locally positive injectivity radius in $N-\Delta$   if and only if $M_n$
has   locally bounded norm of the second fundamental form in $N-\Delta$.
In later sections we will replace the sequence $M_n$ by some subsequence; a key property that
follows from our construction of $M_n$ is that $\Delta$ continues to be the singular set of
convergence for this new
subsequence.

\section{The local geometry around singular points} \label{sec5}

 In this section we study the geometry of $M_n$ nearby points in $\Delta$. Let $q\in\Delta$.
 Then by definition and after possibly  replacing the surfaces $M_n$
by a subsequence, there exists a sequence of points $p_n\in M_n$
such that  $d_{N}(p_{n},q)< 1/n$, where $d_{N}$ is the distance function
in $N$, and $$\frac{1}{nI_n(p_{n})}>n.$$

Consider the continuous functions $h_n\colon M_n \cap \ov{B}_{{N}}(p_{n},1/n)\to \R $ given by
\[
h_n(x)=\frac{d_{N}(x,\partial B_{N}(p_{n},1/n))}{I_{M_n}(x)}.
\]
As $h_n$ vanishes on $M_n\cap \partial B_{N}(p_{n},1/n)$ then there
exists a point $p'_{n}\in M_n\cap B_{{N}}(p_{n},1/n)$ that is a point where $h_n$
takes on its maximum value. This point is said to be a point of {\em almost-minimal injectivity radius}.

Let $\sigma_n:= d_{N}(p'_{n}, \partial B_{N}(p_{n},1/n))$. Then
\begin{equation}\label{inj1}
\frac{\sigma_n}{I_{M_n}(p_n')}=h_n(p_n')\geq h_n(p_n)=\frac{1}{n I_{M_n}(p_n)}>n.
\end{equation}
Let $M'_n$ be the compact surface $M_n\cap \ov{B}_{N}(p'_{n},\sigma_n/2)$.
Note that $\sigma_n$ goes to zero as $n$ goes to infinity and since $M_n$ is compact
with constant mean curvature $H\leq b$, then  for $n$ sufficiently large, $M'_n$
is a compact surface with non-empty boundary that is contained in $\partial B_{N}(p'_{n},\sigma_n/2)$.
Moreover, given $q\in M'_n$ then
\[
\frac{\sigma_n/2}{I_{M_n}(q)}\leq h_n(q)\leq h_n(p_n')=\frac{\sigma_n}{I_{M_n}(p_n')},
\]
which implies that
\begin{equation}\label{inj2}
I_{M_n}(q)\geq \frac{I_{M_n}(p_n')}{2}.
\end{equation}

 Let $I_0$ denote the injectivity radius of $N$. Thinking about the balls in $N$ with
 exponential coordinates, we will consider
 $\ov{B}_N(p'_{n},r)$, $r\leq I_0$, to be the closed ball $\ov{\B}(r)=B_{\rth}(\vec{0},r)\subset \rth$
 of radius $r$ centered at the origin.
After defining $\lambda_n=\frac{1}{I_{M_n}(p'_{n})}$,
 let
 \[
 \wt{M}_n:= \lambda_n M'_n\subset \ov{\B}(\l_n\frac{\sigma_n}{2})
 \subset \rth, \; \partial \wt{M}_n\subset \partial \B(\l_n\frac{\sigma_n}{2}).
 \]
Note that $I_{\wt M_n}(\vec 0)=1$. Since $H_n\leq b$, the constant mean
curvature of $\wt M_n$ goes to zero as
$n$ goes to infinity. By equation~\eqref{inj1}, it follows that $\l_n\frac{\sigma_n}{2}$
goes to infinity as $n$ goes to infinity. By equation~\eqref{inj2}, the sequence of surfaces
$\wt M_n$ has locally positive injectivity radius in $\rth$; indeed given $B$ a compact set
of $\rth$, it follows that for $n$ sufficiently large, for any $q\in \wt{M}_n\cap B$ we
have $I_{\wt{M}_n} (q)\geq 1/2$. However,  since the constant mean curvature of $\wt M_n$
is going to zero as $n$ goes to infinity, it is no longer true for $\wt M_n$ that having
locally positive injectivity radius in $\rth$ is equivalent to having locally bounded norm
of the second fundamental form in $\rth$.

There are two cases to consider. Either  $\wt M_n$ has locally bounded norm of the second
fundamental form in $\rth$ or not. If $\wt M_n$ does NOT have locally bounded norm of the
second fundamental form in $\rth$ then
by Theorem 1.5 in~\cite{mt14}, the surfaces  $\wt M_n$ converge on compact subsets
of $\rth$ to a minimal parking garage
structure of $\rth$ with two oppositely oriented columns, see Figure~\ref{parking}
and see for instance~\cite{mpr14} for a detailed description of this limit object.

\begin{figure}[h]
\begin{center}
\includegraphics[width=5.5in]{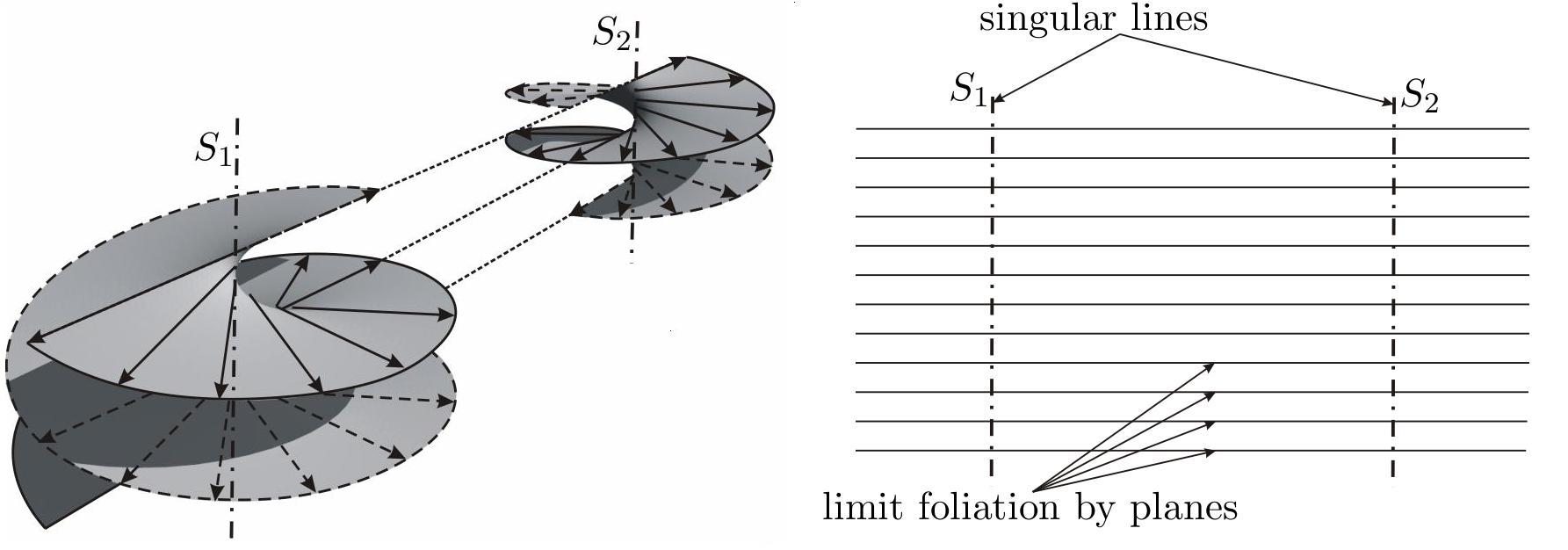}
\caption{Parking garage structure: in this picture, the sequence of surfaces on the
left hand side converge smoothly away from the union $S_1\cup S_2$ or two straight lines orthogonal
to the foliation of horizontal planes described on the right hand side.}
\label{parking}
\end{center}
\end{figure}

If $\wt M_n$ has locally bounded norm of the second fundamental form in $\rth$ then
by applying Theorem 1.3 in~\cite{mt14} and after replacing by a subsequence,
the surfaces $\wt M_n$ converge with multiplicity
one or two on compact
subsets of $\rth$ to a properly embedded minimal surface $\wt M_\infty$, in the sense that any sufficiently
small normal neighborhood of any smooth compact domain $\Omega$ of the limit surface must intersect
$\wt{M}_n$ in one or two components that are small normal graphs over $\Omega$ for $n$ large.
Moreover $\wt M_\infty$ has bounded norm of the second fundamental form, genus at most $g$ and
its injectivity radius at the origin is one. The surface $\wt M_\infty$ either has genus zero
or positive genus. If $\wt M_\infty$ has genus zero and one end, then $\wt M_\infty$ would be a plane or a
helicoid, see~\cite{mr8} and also~\cite{bb1}. This cannot happen because the  injectivity
radius of  $\wt M_\infty$  at the origin is one.
Therefore if $\wt M_\infty$ has genus zero, then it must have
more than one end and there are two cases: $\wt M_\infty$ is either a
catenoid (finite topology~\cite{lor1}), or $\wt M_\infty$ is a Riemann minimal
example (infinite topology~\cite{mpr6}); see Figure~\ref{riemann} and see for
instance~\cite{mpr6} for a detailed description of a Riemann minimal example.
 \begin{figure}[h]
\begin{center}
\includegraphics[width=2.1in]{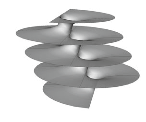}
\caption{Riemann minimal example.}
\label{riemann}
\end{center}
\end{figure}

In summary, in the limit we obtain one of the examples listed below:
\ben
\item a catenoid or a Riemann minimal example;
 \item  a minimal parking garage structure  with two oppositely oriented columns;
\item a properly embedded minimal surface with positive genus at most $g$.
\een

\begin{proposition}\label{limitsurface}
The surfaces $\wt{M}_n$ converge with multiplicity one or two on
compact subsets of $\rth$ to a catenoid or to a
properly embedded minimal surface with  positive genus at most $g$.
\end{proposition}

\begin{proof}
In light of the previous discussion, we need to rule out the occurrence
of a limit surface that is a Riemann
minimal example or a minimal parking garage structure of $\rth$ with two oppositely
oriented columns. We next rule out the case that a Riemann minimal example occurs.
An analogous discussion rules out the possibility that a minimal parking garage structure
with two oppositely oriented columns occurs; see Remark~\ref{remark:park}.

Suppose that the limit object is  a Riemann
minimal example which we denote by $\cR$.
The surface $\cR$ is a properly
embedded minimal planar domain in $\rth$ of infinite topology which is foliated by circles and
lines in a family of parallel planes. Given $R>0$ let $\Omega(R):= \B(R)- \cR$.
Then, given $m\in \mathbb N$, there exists $R>0$ such that $\B(R)\cap \cR$
has at least $2m$ boundary components in $\partial \B(R)$ such that the following holds.
At least $m$ of these boundary components do not bound a disk whose interior is contained
in $\Omega(R)$ and each pair of curves in this collection of boundary components does not
bound an annulus whose interior is contained in $\Omega(R)$, see Figure~\ref{riemcurves}.
 \begin{figure}[h]
\begin{center}
\includegraphics[width=2in]{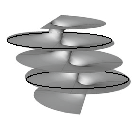}
\caption{These curves are homotopically non-trivial and do not bound an annulus.}
\label{riemcurves}
\end{center}
\end{figure}
Thus, by the nature of the convergence, a geometric picture that is similar to the
one just  described for $\cR$ is true for $\wt M_n$. Namely, let $\wt M_n(R)$ denote
the connected component of $\wt M_n\cap \B(R)$ that contains the origin.
Given $m\in\mathbb N$ there exists $R>0$ such that for $n$ sufficiently
large $\partial \wt M_n(R)$ contains at least $m$ simple closed curves

\begin{equation}\label{curves}
\{\Gamma_1(n),\dots,\Gamma_m(n)\}
\end{equation}
and the following holds.
\begin{itemize}
\item Each $\Gamma_i(n)$, $i:=1,\dots m$ does not bound a disk whose interior
is contained in $\B(R)- \wt M_n(R)$.
\item Each pair of curves $\{\Gamma_i(n),\Gamma_j(n)\}$, $i\neq j$,  does not
bound an annulus whose interior is contained in $\B(R)- \wt M_n(R)$.
\end{itemize}
\begin{remark} \label{remark:park} {\em The same description holds when the
case of  a minimal parking garage structure
with two oppositely oriented columns occurs.
In this case, the simple closed curves $\Gamma_i(n)$, $i:=1,\dots m$, wind once around
the pair of columns $S_1\cup S_2$  in
Figure~\ref{parking}. }\end{remark}

Let $\{\gamma_1(n),\dots,\gamma_m(n)\}$ be the collection of curves in $M_n$
corresponding to the curves $\{\Gamma_1(n),\dots,\Gamma_m(n)\}$ in $\partial \wt M_n(R)$.

\begin{claim}\label{nontriv}
For $n$ sufficiently large, the curve $\gamma_i(n)$, $i=1,\dots m$, is NOT homotopically  trivial in $M_n$
and each pair of curves $\{\gamma_i(n),\gamma_j(n)\}$, $i\neq j$, does NOT bound an annulus in $M_n$.
\end{claim}

\begin{proof}
Recall that the origin in $\rth$ corresponds to points $p'_n\in M_n$ and let $B(\ve)$
denote the balls of radius $\ve$ in $N$ centered at $p'_n$. Recall that
by item~\ref{ait1} of Theorem~\ref{separate}, $M_n$ separates $N$ into two closed regions
and one of them, denoted by $G_{M_n}$ is mean convex.
 By the previous discussion there exists $\ve_n>0$ such
 that $\gamma_i(n)\subset \partial B(\ve_n)$, $\gamma_i(n)$ is homotopically
non-trivial in $\ov{B}(\ve_n)\cap G_{M_n}$ and each pair of curves $\{\gamma_i(n),\gamma_j(n)\}$, $i\neq j$,
 does not bound an annulus in $\ov{B}(\ve_n)\cap G_{M_n}$.  Moreover, $\lim_{n\to\infty}\ve_n=0$.

Suppose that $\gamma_i(n)$ is homotopically trivial in  $M_n$, namely that $\gamma_i(n)$
bounds a disk in $M_n$. Then, by the results in~\cite{my2}, solving the annular
Plateau problem for $\gamma_i(n)$ in $G_{M_n}$ produces an embedded least-area
disk $D_n\subset G_{M_n}$ with  $\partial D_n=\gamma_i(n)\subset \partial B(\ve_n)$.
By construction, $\partial D_n$ is homotopically non-trivial in $B(\ve_n)\cap G_{M_n}$
and thus $D_n$ cannot be contained in $B(\ve_n)\cap G_{M_n}$. Hence, the disk $D_n$
lifts to a minimal disk in $\rth$ that is not contained in $\B(\ve_n)$ but with
boundary contained in $\partial \B(\ve_n)$.  This violates the mean convex hull
property for minimal surfaces and gives a contradiction, which proves that $\gamma_i(n)$
is homotopically  non-trivial in  $M_n$.

Suppose that the pair of curves $\{\gamma_i(n),\gamma_j(n)\}$, $i\neq j$, bounds an
annulus in $M_n$. Since these curves are homotopically non-trivial, solving the annular Plateau's problem
in $G_{M_n}$ (see Dehn's Lemma for Planar Domains in~\cite{my1}   as adapted
by the more general boundary conditions in~\cite{my2})
for $\{\gamma_i(n),\gamma_j(n)\}$, $i\neq j$ yields an embedded least-area
annulus $A_n\subset G_{M_n}$ with  $\partial A_n=\gamma_i(n)\cup\gamma_j(n)\subset \partial B(\ve_n) $.
By construction, $ A_n$ is not contained in $B(\ve_n)\cap G_{M_n}$. Let $\B(z_k, \ve_n)$
denote a lift of $B(\ve_n)$ into $\rth$ and note that if $k_1\neq k_2$
then $d_{\rth}(z_{k_1},z_{k_2})\geq I_0$. Since $\ve_n<I_0$ and
$\gamma_i(n), \gamma_j(n)\subset \partial B(\ve_n)$, each closed curve lifts to
a closed curve in  $\partial \B(z_k, \ve_n)$. Abusing the notation, let  $\gamma_i(n)$
denote the lift of $\gamma_i(n)$ in $\rth$ that is contained in $\B(z_1,\ve_n)$.
Let  $\wt{A}_n$ be the annulus that is the lift of $ A_n$    that has $\gamma_i(n)$
as one of its boundary component; note that $\wt{A}_n$ is not contained in $\B(z_1,\ve_n)$.
Suppose that the other boundary component of $\wt{A}_n$ is contained in a
certain $\B(z_k,\ve_n)$ where $z_k\neq  z_1$. This being the case, the boundary
of $\wt{A}_n$ consists of two components $\gamma_i(n)$ and $\gamma_j(n)$ with
$\gamma_i(n)\subset \B(z_1,\ve_n)$, $\gamma_j(n)\subset \B(z_k,\ve_n)$, $d_{\rth}(z_1,z_{k})\geq I_0$
and $\lim_{n\to\infty}\ve_n=0$.
This violates a well-known property of minimal surfaces (pass a
catenoid between the two balls until a first point of interior contact). This shows that
the other boundary component of $\wt{A}_n$ must also be contained in $\B(z_1,\ve_n)$.
This being the case, the fact that $\wt{A}_n$ is not contained in $\B(z_1,\ve_n)$
violates the mean convex hull property for minimal surfaces. Thus, we have obtained
a contradiction and completed the proof of the claim.\end{proof}

In order to finish the proof of the proposition, it suffices to show that if a Riemann
minimal example occurs, then there are two curves $\gamma'_1(n)$ and $\gamma'_2(n)$ as
described in the previous claim and which bound an annulus in $M_n$. This will be a
simple consequence of the following lemma.

\begin{lemma}\label{genus}
Let $\Sigma$ be a closed, possibly disconnected, surface of positive genus $g$  and let $\Gamma$ be a
collection of simple closed curves in $\Sigma$ that are homotopically non-trivial and
pair-wise disjoint. If the number of curves in $\Gamma$ is greater than $3g-2$, then
there exists  a pair of curves in $\Gamma$ that bounds an annulus in $\Sigma$.
\end{lemma}
\begin{proof}
 Let $\Gamma:=\{\gamma_1,\dots, \gamma_g, \dots, \gamma_{g+k}\}$, $k>2g-2$,
be a collection of simple closed curves in $\Sigma$ that are homotopically non-trivial
and pair-wise disjoint. Without loss of generality, we may assume that
every component of $\Sigma$ contains an element of $\Gamma$; in particular,
we may assume that $\Sigma$ has no spherical components. If each component
of $\Sigma$ has genus 1, namely it is a 3-torus, then since
$g+k> g$,
the number of curves in $\Gamma$ is greater then the number of components. Therefore,
at least one component contains two elements of $\Gamma$ and thus the lemma holds
in this case. Similar reasoning demonstrates that
it suffices to prove the lemma under the additional hypothesis that
none of the  components of $\Sigma$ that
contains only one element of $\Gamma$ is a 3-torus. Otherwise, we would remove
such a component from $\Sigma$ and such an element from the collection.

By definition of genus, $\Sigma-[\bigcup_{i=1}^{g+k}\gamma_i]$
consists of at least $k+1$ connected components.
Namely, $\Sigma-[\bigcup_{i=1}^{g+k}\gamma_i]=\Sigma_1\cup\dots\cup \Sigma_n$ with $n\geq k+1$.
Suppose that none of these components is an annulus. Then, for each $i=1,\dots,n$, $\chi(\Sigma_i)\leq -1$,
where $\chi$ denotes the Euler characteristic function. Thus,
\[
2-2g\leq\chi(\Sigma)=\chi\left (\bigcup_{i=1}^n\Sigma_i\right)=\sum_{i=1}^n \chi(\Sigma_i)\leq -n.
\]
Therefore, $n\leq 2g-2$ which gives that $g+k\leq g+n-1\leq 3g-3$. In other words, if the
number of curves in $\Gamma$ is greater than $3g-3$, then at least one component
of $\Sigma-[\bigcup_{\gamma\in\Gamma}\gamma]$ is an annulus,
and by our previous discussion, the boundary curves of this
annulus are distinct elements in $\G$.
This completes the proof of the claim.
\end{proof}

In light of Lemma~\ref{genus}, if a Riemann minimal example occurs, let $n$ be sufficiently
large so that the number of curves in equation~\eqref{curves} is greater than $3g-2$.
By Claim~\ref{nontriv} such curves are homotopically non-trivial in $M_n$ and thus, using
Lemma~\ref{genus} gives that at least two of them bound an annulus in $M_n$. This contradicts
Claim~\ref{nontriv} and finishes the proof of the proposition.
\end{proof}

\section{Global properties due to singular points} \label{sec6}

In this section we study what implications the presence of  points in the singular set of convergence
$\Delta$ for the sequence $M_n$  has for
the global geometry of the surfaces $M_n$ nearby a fixed point in $\Delta$.

Let $q\in\Delta$ be a singular point. By the previous
section (see the discussion after equation~\eqref{inj2}) and after replacing
by a subsequence, there exists a sequence
of points $p'_n$ converging to $q$ and sequences of numbers $\delta_n$, $\rho_n$
converging to zero, with $\lim_{n\to\infty}\frac{\rho_n}{\delta_n}=\infty$,
such that $\wt M_n=\frac1{\delta_n}[\B(\rho_n)\cap M_n]$ converges (with multiplicity one or two)
on compact subsets of $\rth$ to
either a catenoid $\mathcal C$ or a properly embedded minimal surface with positive genus.
When only the former happens, we say that $q$ is a {\em catenoid singular point}. Note that since
the genus is additive and the genus of $M_n$ is bounded by $g$ independently of $n$,  after replacing
the sequence $M_n$
by a subsequence, the
number of singular points that are not catenoid singular points is at most $g$. Indeed we
will show later in  Proposition~\ref{singnogenus} that all singular points
are catenoid singular points.

Throughout the rest of the section we will
deal with catenoid singular points and will assume that there are at most $g$ singular points
that are not of catenoid-type.
Let $q\in\Delta$ be a catenoid singular point and let $\cC$ denote a limit
catenoid related to $q$ with related points $p'_n$ converging to $q$ as described
in the previous paragraph. Let $l_\cC$ denote the line in $\rth$ such
that $\cC$ is rotationally invariant around that line, and let $\Pi_\cC$ denote the plane
perpendicular to $l_\cC$ that is a plane of symmetry for $\cC$. Let $\gamma_\cC$ denote the
geodesic circle in $\cC$ that is obtained by intersecting $\Pi_\cC$ with $\cC$ and
let $c_{\mathcal C}$ denote the point that is the center of this circle.
We can also associate a sequence of simple closed
curves $\gamma_n(q)\subset M_n$ corresponding to the curves in $\wt M_n\cap \Pi_\cC$
that are converging to $\gamma_\cC$ (if the multiplicity of convergence is two then we make a
choice for one of the two almost-circles in $\wt M_n\cap \Pi_\cC$ that give rise to the
limit $\cC$). We call the curve $\gamma_n(q)$ a {\em singular loop at $q$}
and we denote the points in $N$ corresponding to $c_{\mathcal C}$ by $c_n(q)$ and we refer to $c_n(q)$
as the {\em center of $\gamma_n(q)$}.

Recall that given a catenoid $\cC$, the flux of $\gamma_\cC$ is non-zero. Therefore,
by the nature of the convergence, the flux of $\gamma_n(q)$ is also non-zero. Since the
flux is a homological invariant, this implies that $\gamma_n(q)$ is homotopically non-trivial in $M_n$.

Suppose $q_1, q_2\in\Delta$ are two distinct catenoid singular points that are the limit of respective
points $p'_n(1),p'_n(2)\in  M_n$, as described in the previous paragraphs, and
let $\gamma_n(q_1),\gamma_n(q_2)$ be sequences of singular loops at $q_1$ and $q_2$.
Suppose that $\gamma_n(q_1)\cup\gamma_n(q_2)$ is the boundary of an annulus $\wt A_n(q_1,q_2)$ in $M_n$.
Since the lengths of $\gamma_n(q_1)$ and $\gamma_n(q_2)$ are going to zero, the singular
loops lift to closed curves in $\rth$ and the annulus $\wt A_n(q_1,q_2)$ lifts to an
annulus $A_n(q_1,q_2)$ in $\rth$.  Abusing the notation, we denote the   boundary
of $A_n(q_1,q_2)$ by $ \gamma_n(q_1)\cup  \gamma_n(q_2)$ and let $c_n(q_1)$ and $c_n(q_2)$
be points in $\rth$ corresponding to  the centers of $\gamma_n(q_1)$ and $\gamma_n(q_2)$ and
related to the chosen lifts. Let $l_n(q_1,q_2)$ be the straight line containing $c_n(q_1)$
and $c_n(q_2)$ and let $C_n(q_1,q_2,R)$ denote the cylinder of radius $R$ around $l_n(q_1,q_2)$.
Thus, by construction and  for $n$ sufficiently large,
$\gamma_n(q_1)\cup\gamma_n(q_2)\subset C_n(q_1,q_2,\delta_n)$ but $A_n(q_1,q_2)$
is NOT contained in $C_n(q_1,q_2,\delta_n)$. The mean curvature comparison principle
shows that $A_n(q_1,q_2)$  is not contained in $C_n(q_1,q_2,\frac1{2H_n})$.

Let $z_n$ be a point in $A_n(q_1,q_2)$ that is farthest away from $l_n(q_1,q_2)$.
To simplify the notation, after applying a sequence of translations of $\mathbb{R}^3$, assume
that the origin $\vec 0\in l_n(q_1,q_2)$ and the projection
of $z_n$ onto $l_n(q_1,q_2)$ is the origin. Let $r_{z_n}$ be the
ray $\{s\frac{z_n}{|z_n|}\mid s>0\}$ and for $t\in (0,|z_n|]$, let
$\Pi(z_n)_t$ be the plane perpendicular to $r_{z_n}$ at the point $t\frac{z_n}{|z_n|}$.
Note that the mean curvature vector at $z_n$ is $-H_n \frac{z_n}{|z_n|}$.
Abusing the notation, let $A_n(q_1,q_2)$ denote the connected component of
$A_n(q_1,q_2)-\Pi(z_n)_{\delta_n}$ that contains $z_n$ and let $\cH_n$ denote
the half-space of $\rth$ that contains $z_n$ and has $\Pi(z_n)_{\delta_n}$ as
its boundary. Since $\cH_n$ is simply connected  and $\partial A_n(q_1,q_2)$ is
contained in $\Pi(z_n)_{\delta_n}$, $A_n(q_1,q_2)$ separates $\cH_n$ into two
components. One of these component is bounded and we denote its closure
by $G_{A_n}$. Since the mean curvature vector at $z_n$ is $-H_n \frac{z_n}{|z_n|}$ and $z_n$
is a point in $A_n(q_1,q_2)$ that is furthest away from  $\Pi(z_n)_{\delta_n}$,
then $G_{A_n}$ is mean convex. Recall that by Theorem~\ref{separate} and its proof, the component $M'_n$
of $\Pi^{-1}(M_n)$ in $\rth$ that contains $A_n(q_1,q_2)$ separates $\rth$ into two components.
One of these components is mean convex and we denote its closure by $G_{M'_n}$.
Let $W_n=G_{A_n}\cap G_{M'_n}$. Note that $z_n\in \partial W_n$ is a point of $W_n$
that is furthest away from $\Pi(z_n)_{\delta_n}$.

A standard application of the Alexandrov reflection principle to the compact
mean-convex region $W_n$, using the family of
planes $\Pi(z_n)_t$, gives that the connected component $A^+_n(q_1,q_2)$ of
$A_n(q_1,q_2)- \Pi(z_n)_{\frac{\delta_n+|z_n|}{2}}$
containing $z_n$ is graphical over its projection to $\Pi(z_n)_{\frac{\delta_n+|z_n|}{2}}$
and the reflected image $A^-_n(q_1,q_2)$ of $A^+_n(q_1,q_2)$ in the
plane $\Pi(z_n)_{\frac{\delta_n+|z_n|}{2}}$ intersects
$M_n'$ only along the boundary of $A^+_n(q_1,q_2)$.
 Since $\delta_n$ is going to zero as $n$ goes to infinity, and $|z_n|\geq \frac1{2H_n}\geq\frac1{2b}$,
 then  we can assume that
 \begin{equation}\label{equations}
 |z_n|-\frac1{6b}>\frac{\delta_n+|z_n| }2,
 \end{equation}
 that is the distance from $z_n$ to the plane $\Pi(z_n)_{\frac{\delta_n+|z_n|}{2}}$ is at least $\frac 1{6b}$.

 Let $A^*_n(q_1,q_2)$ be the connected component of $A^+_n(q_1,q_2)- \Pi(z_n)_{|z_n|-\frac 1{12b}}$
 that contains $z_n$. By construction, a point in $A^*_n(q_1,q_2)$ is a point in $A^+_n(q_1,q_2)$
 at distance at least $\frac 1{12b}$ from the boundary of $A^+_n(q_1,q_2)$. Thus,
 applying the uniform curvature
estimates in~\cite{rst1} for oriented  graphs with constant mean
curvature (graphs are stable with curvature estimates away from their boundaries), gives that
points in $A^*_n(q_1,q_2)$ satisfy a uniform curvature estimate.

Furthermore, the same standard application of the Alexandrov reflection principle implies also
the following. Let $G(q_1,q_2)$ be the bounded open region of $\rth$ contained
between  $A^*_n(q_1,q_2)$ and its reflection across  the plane $\Pi(z_n)_{|z_n|-\frac 1{12b}}$.
Then $G(q_1,q_2)$ is contained in the interior of $W_n$ and there exists $\ve_1>0$
such that $\textrm{Volume}(G(q_1,q_2))>\ve_1$ that is
independent of $n$ (for $n$ sufficiently large). In particular, if $\wt G(q_1,q_2)$
denotes the image of $G(q_1,q_2)$ into $N$ via the universal covering map, then $\wt G(q_1,q_2)$ is
contained in $G_{M_n}$, that is the closure of the mean convex component of $T-M_n$,
and  $\textrm{Volume}(\wt G(q_1,q_2))>\ve_1$. Moreover,  if $\wt G(q_1,q_2)$
and  $\wt G(p_1,p_2)$ are regions of $N$ related to two distinct annuli, then these regions are disjoint.

\section{Bounding the number of singular points}\label{secboundsing}

In this section we bound the number of  points in $\Delta$. Since the number of singular points that
are not catenoid singular points is at most $g$, it suffices to bound the number of catenoid singular points.

Let $\{q_1,\dots, q_m\}\in \Delta$ be a collection of catenoid singular points.
It is important to remark that in what follows, the integer $n\in \N$
is chosen sufficiently large so that the estimates
of previous sections, such as those in appearing in \eqref{equations},
make sense at each of the points in $\{q_1,\dots, q_m\}$.
By definition and the discussion in the previous section,
to each $q_i$ corresponds a sequence of singular loops $\gamma_n(q_i)$ and  such loops are
homotopically non-trivial. Thus, by applying Claim~\ref{genus}, if $m>k(3g-2)$  we obtain at least
$k$ annuli $A_1,\dots, A_k$ with pairs of singular loops as their boundaries. Note
that if $A_i\cap A_j\neq \emptyset$ then their intersection must be an annulus with
a pair of singular loops as its boundary. Therefore, after possibly replacing the collection
of annuli $A_1,\dots, A_k$ with a different collection, we can assume that the annuli are pairwise disjoint.

For every $i=1,\dots , k$, let $\wt G_i\subset G_{M_n}$ denote the region of $N$
related to $A_i$ and obtained by applying the Alexandrov reflection principle, as
described in the last paragraph of the previous section.  Recall that there exists $\ve_1>0$,
independent of $n$ and $i$ such that $\textrm{Volume}(\wt G_i)\geq \ve_1$ and
that $\wt G_i \cap \wt G_j=\O$ if $i\neq j$. Then we have the following inequality:
\[
k\ve_1\leq \sum^{k}_{i=1}\textrm{Volume}(\wt G_i)
=\textrm{Volume}(\bigcup^{k}_{i=1}\wt G_i)\leq  \textrm{Volume}(N).
\]
Therefore,
\[
k\leq  \frac{ \textrm{Volume}(N)}{\ve_1}
\]
which implies that
 \[
m\leq  \frac{ \textrm{Volume}(N)}{\ve_1}(3g-2).
\]

Thus, adding also the bound for the number of singular points that are not catenoid singular
points, the previous inequality gives that the number of singular points is bounded by
\[
 \frac{ \textrm{Volume}(N)}{\ve_1}(3g-2)+g.
\]

\begin{remark}\label{rmkboundsing} {\em
Note that the proof that the number of singular points is bounded does NOT use the fact
that the area of $M_n$ is becoming arbitrarily large.}
\end{remark}

\section{The final contradiction} \label{sec:final-cont}

In this section we prove that   the area of $M_n$ is uniformly bounded from above. This
contradicts the fact that $\textrm{Area}(M_n)>n$ and this contradiction
will finish the proof of Theorem~\ref{area2}.

Let $\Delta:=\{q_1, \dots, q_m\}$ be the set of singular points. The results
in the previous section give that
\[
m\leq \frac{ \textrm{Volume}(N)}{\ve_1}(3g-2)+g.
\]
Note that since $M_n$ separates $N$ and the norms of the second fundamental forms of $M_n$
are uniformly bounded on compact sets of $N-\Delta$, by applying
Proposition~\ref{area51} the following holds. If $p\in N-\Delta$ and $\ve>0$
is such that $B_N(p,\ve)\cap\Delta=\emptyset$, then there exists a constant $T(\ve)$ such
that $\textrm{Area}(M_n \cap B_N(p,\frac\ve2))<T(\ve)$. A standard compactness argument
then gives that there exists a surface $M_\infty$ properly immersed in $N-\Delta$ such that,
up to a subsequence, $M_n-\Delta$  converges to $M_\infty$ on compact subsets of $N-\Delta$.
The surface $M_\infty$ has constant mean curvature $H$, for a certain $H\in [a,b]$.
Since $\Delta$ is finite, there exists $r>0$ such that for any $q\in\Delta$,
$B_N(q,r)\cap \Delta=q$ and $B_N(q,\tau)\cap M_\infty\neq \O$,  for any $\tau\in (0,r]$.

  \begin{claim}\label{convalex}
The sequence $M_n-\Delta$  converges to $M_\infty$ with multiplicity one and $M_\infty$
is strongly Alexandrov embedded in $N-\Delta$.
\end{claim}
\begin{proof}
Recall that $M_n$ separates $N$ into two regions and one of them, denoted by $G_{M_n}$,
is mean convex. Given $p\in M_\infty$, then there exists $\ve>0$ such that a pointed connected
  component of $B_N(p,\ve)\cap M_\infty$, which we denote by $\Omega(p)$, is a graph over the
  tangent plane to $M_\infty$ at $p$, $T_pM_\infty$, and it is the limit of a sequence of
  graph $U_n(p)\subset M_n$ over $T_pM_\infty$. Note that $\Omega(p)$ has a well-defined
  mean curvature vector obtained as the limit of the mean curvature vectors of $U_n(p)$.
  If $M_n$ contained more than one graph over  $T_pM_\infty$ converging to $\Omega(p)$,
  since $M_n$ separates $N$, then the mean curvature vectors would change orientation on
  consecutive graphs in $M_n$ and the mean curvature vector on $\Omega(p)$ would not
  be well-defined. This proves that $M_n-\Delta$  converges to $M_\infty$ with multiplicity one.

  By the previous argument and the fact that the surfaces $M_n-\Delta$ converge
  to $M_\infty$ with multiplicity one, then  the connected open regions that
  are $\textrm{Int}(G_{M_n})-\Delta$ converge to an open region $W$ in $N-\Delta$
  and $\partial W=M_\infty$. This shows that $M_\infty$ is strongly Alexandrov embedded
  in $N-\Delta$, which finishes the proof of the claim.
\end{proof}

By the nature of the convergence with multiplicity one, and since $M_\infty$ is properly
immersed in $N-\Delta$, for any $\ve>0$, there exists $K(\ve)$ such that
$$\lim_{n\to\infty}\mathrm{Area}(M_n\cap[N-\bigcup^m_{i=1}B_N(q_i,\ve)] )
=\mathrm{Area}(M_\infty\cap[N-\bigcup^m_{i=1}B_N(q_i,\ve) ])<K(\ve).$$

Fix $\ve>0$ such that $4e^{-\ve b}\geq 2$, where $b$ is the upper bound for
the mean curvature of $M_n$, $B_N(q_i,2\ve)$ is an open ball in $N$ and for any $i,j\in \{1,\dots, m\}$
with $i\neq j$, then $B_N(q_i,2\ve)\cap B_N(q_j,2\ve)=\emptyset$. Then, by the previous
argument, for each $i:=1,\dots, m$ and $n$ sufficiently large,
$$\mathrm{Area}(M_n\cap[N-\bigcup^m_{i=1}B_N(q_i,\ve) ])<K(\ve)+1.$$

Recall that $H_n\leq b$. By the monotonicity formula for $H_n$-surfaces, see for
instance~\cite{si1}, it follows that

$$\frac{\mathrm{Area}(M_n\cap B_N(q_i,2\ve))}{4\ve^2}
\geq e^{-\ve H_n}\frac{\mathrm{Area}(M_n\cap B_N(q_i,\ve))}{\ve^2}
\geq e^{-\ve b}\frac{\mathrm{Area}(M_n\cap B_N(q_i,\ve))}{\ve^2}$$

This implies that

$$\mathrm{Area}(M_n\cap B_N(q_i,2\ve))\geq4e^{-\ve b} \mathrm{Area}(M_n\cap B_N(q_i,\ve))
\geq2\mathrm{Area}(M_n\cap B_N(q_i,\ve)). $$

Therefore, for $n$ sufficiently large,

\[
\begin{split}
\mathrm{Area}(M_n\cap B_N(q_i,\ve))&\leq\mathrm{Area}(M_n\cap B_N(q_i,2\ve))
-\mathrm{Area}(M_n\cap B_N(q_i,\ve))\\
&=\mathrm{Area}(M_n\cap [B_N(q_i,2\ve)- B_N(q_i,\ve)])\\
&<K(\ve)+1.
\end{split}
\]

Finally, this implies that for $n$ large,
\[
\begin{split}
\mathrm{Area}(M_n)=& \mathrm{Area}(M_n\cap[N-\bigcup^m_{i=1}B_N(q_i,\ve) ])
+\sum_{i=1}^m\mathrm{Area}(M_n\cap B_N(q_i,\ve))\\
<&(m+1)(K(\ve)+1).
\end{split}
\]
Since $\ve$ is fixed, independent of $n$, and $m$
is at most $ \frac{ \textrm{Volume}(N)}{\ve_1}(3g-2)+g$,
this contradicts the fact $\textrm{Area}(M_n)>n$.
This contradiction completes the proof of Theorem~\ref{area2}.
\end{proof}

\subsection{The proof of Theorem~\ref{area2} for disconnected $H$-surfaces $M$} \label{sec:8.1}
We next explain
why Theorem~\ref{area2} holds for another choice
of constant $A'(N,a,b,g)$ when $M$ in its statement is not necessarily connected.
First we may assume that the
constants $A(N, a,b, g)$ given in
Theorem~\ref{area2} are increasing as a function of the genus $g$. Also notice
that any flat 3-torus $T$ with injectivity radius $I_0$ and diameter $d$
has a fixed upper bound $V(I_0,d)$ on its volume.  Observe that any collection of pairwise
 disjoint embedded $H$-spheres in such a $T$ with $H\in [a,b]$ bound a family of pairwise disjoint
balls in $T$ with volume at least $\frac43 \pi  a^3$ and the sum
of these volumes is bounded by  $V(I_0,d)$. Thus,
the number of spherical components of a disconnected closed $H$-surface $\Sigma$ of genus $g$
in $T$ is bounded by some  constant $S(I_0,d,a)\in \N$.  Therefore
if $M$ is a possibly {\em disconnected} surface satisfying the
other hypotheses of Theorem~\ref{area2}, then it can have at
most $S(I_0,d,a) +g$ components.  Hence, the area of $M$
is at most $$A'(N,a,b,g)=[S(I_0,d,a)+g]A(N,a,b,g),$$ which proves our desired claim.

\section{Compactness of $H$-surfaces in a flat 3-torus}\label{compactness}

In this section, we prove Theorem~\ref{seq-compact} from the Introduction. Let $N$ be
a flat 3-torus and let $M_n$ be a
sequence of closed $H_n$-surfaces  in $N$, $H_n\in[a,b]$, of genus at most $g$.
Theorem~\ref{area2} and the discussion at the end of the previous section gives
that there exists a constant $C$ independent of $n$ such that
\[
\sup_n\textrm{Area}(M_n)<C.
\]

By the results in Section~\ref{secboundsing} and in particular note Remark~\ref{rmkboundsing},
after passing to a subsequence, there exists a possibly empty finite set of points
$\Delta$, namely the set of singular points of convergence, such that $M_n$ has locally
bounded norm of the second fundamental form in $N-\Delta$.

The standard compactness argument already used   in Section~\ref{sec:final-cont},
see  Claim~\ref{convalex}, gives that there
exists a surface $M_\infty$ strongly Alexandrov embedded in $N-\Delta$ such that, up to a
subsequence, $M_n-\Delta$  converges to $M_\infty$ on compact subsets of $N-\Delta$ with multiplicity one.
The surface $M_\infty$ has constant mean curvature $H$, for a certain $H\in [a,b]$.
Moreover the convergence to $M_\infty$ has multiplicity one which implies that the
genus of $M_\infty$ is at most $g$.  Recall that $\Delta$ is in the closure of $M_\infty$.

\begin{claim}
The points in $\Delta$ are
removable singularities for $M_\infty$.
\end{claim}
\begin{proof}
By Theorem~\ref{area2} and the discussion at the end of the previous section, there
exists a constant $C>0$ such that $\sup_n\textrm{Area}(M_n)<C$.
Gauss-Bonnet Theorem together with the Gauss equation gives that
\[
2\pi \chi(M_n)=\int_{M_n} K_{M_n}=\int_{M_n}(2H_n^2-\frac{|A_{M_n}|^2}2).
\]

Since the genus of $M_n$ is at most $g$, $H_n\leq b$ and $\textrm{Area}(M_n)<C$ then,
using the previous inequality gives that there exists a constant $C_a$, independent of $n$, such that
\[
\int_{M_n}|A_{M_n}|^2<C_a.
\]

Since $\int_{M_n}|A_{M_n}|^2=\int_{M_n-\Delta}|A_{M_n}|^2$, by the nature of the
convergence with multiplicity one it follows that
\[
\int_{M_\infty}|A_{M_\infty}|^2\leq C_a.
\]
By applying a rescaling argument around each point $q\in\Delta$, this gives that
there exists a constant $C_b>0$ such that
for any $p\in M_\infty$,
\[
|A_{M_\infty}|(p)<\frac{C_b}{d_N(p,\Delta)}.
\]
Since $M_\infty$ is a weak $H$-lamination of $N-\Delta$,
Theorem~1.2 in~\cite{mpr21} implies that $M_\infty$
extends smoothly across $\Delta$ to a weak $H$-lamination of $N$ and the points in $\Delta$ are
removable singularities for $M_\infty$.
\end{proof}

Since $M_\infty$ extends across $\Delta$, if
by abusing the notation
we denote by $M_\infty$ the related surface $M_\infty\cup\Delta$, then $M_\infty$
is strongly Alexandrov embedded in $N$.

It remains to show that the singular set of convergence $\Delta$ is contained in the set of
points of self-intersection of $M_\infty$. Note that if $p\in M_\infty$ is not a point of
self-intersection of $M_\infty$, then
\begin{equation}\label{density1}
\lim_{r\to 0}\frac{\textrm{Area}(M_\infty\cap [B_N(p,2r)-B_n(p,r)])}{\pi r^2}=3.
\end{equation}
Instead, by the description at point of self-intersection, if $p\in M_\infty$ is a point of
self-intersection of $M_\infty$, then
\begin{equation}\label{density2}
\lim_{r\to 0}\frac{\textrm{Area}(M_\infty\cap [B_N(p,2r)-B_n(p,r)])}{\pi r^2}=6.
\end{equation}
Recall that by the monotonicity formula for $H_n$-surfaces, since $H_n\leq b$, it follows
that for any $q\in M_n$ and $r_1<2r_2<I_0$ the following holds,
\[
\begin{split}
\frac{\mathrm{Area}(M_n\cap B_N(q,r_2))}{\pi r_2^2}&\geq e^{-(r_2-r_1) H_n}
\frac{\mathrm{Area}(M_n\cap B_N(q,r_1))}{\pi r_1^2}\\ &
\geq e^{-r_2 b}\frac{\mathrm{Area}(M_n\cap B_N(q,r_1))}{\pi r_1^2}.
\end{split}
\]
Therefore, applying this inequality, we obtain that
\begin{equation}\label{8/3}
\begin{split}
\frac{\mathrm{Area}(M_n\cap [B_N(q,2r_2)- B_N(q,r_2)])}{\pi r_2^2}&
\geq (4 e^{-r_2 b}-1)\frac{\mathrm{Area}(M_n\cap B_N(q,r_2))}{\pi r_2^2}\\
&\geq  (4 e^{-r_2 b}-1)e^{-r_2 b} \frac{\mathrm{Area}(M_n\cap B_N(q,r_1))}{\pi r_1^2} \\ &
\geq   \frac 83 \frac{\mathrm{Area}(M_n\cap B_N(q,r_1))}{\pi r_1^2},
\end{split}
\end{equation}
if $r_2$ is chosen so that $4 e^{-r_2 b}-1\geq  \frac 83$.

Let $q\in \Delta$ be a singular point. If $q$ is a catenoid singular point then there
exists $q_n\in M_n$ and $\delta_n>0$ such that $\lim_{n\to\infty}q_n=q$, $\lim_{n\to\infty}\delta_n=0$ and
\begin{equation}\label{genuszero}
\frac{\mathrm{Area}(M_n\cap B_N(q_n,\delta_n))}{\pi \delta_n^2}\geq \frac 32.
\end{equation}
If $q\in \Delta$ is a singular point that is NOT a catenoid singular point then the
limit surface given by Proposition~\ref{limitsurface} is a properly embedded minimal
surface with positive genus at most $g$. Let $M_\infty$ denote this limit surface.
There are three cases to consider. If $M_\infty$ has one end, then it is a helicoid
with a finite number of handles attached to it, see~\cite{bb1,mpe3,mr8}. If $M_\infty$
has finite topology and more than one end, then it has at least three ends~\cite{sc1}.
If $M_\infty$ has infinite topology, then it is a Riemann minimal example with a finite
number of handles attached to it, see~\cite{mpr6}. If either of the three cases happens, then there
exists $q_n\in M_n$ and $\delta_n>0$ such that $\lim_{n\to\infty}q_n=q$, $\lim_{n\to\infty}\delta_n=0$ and
\begin{equation}\label{positivegenus}
\frac{\mathrm{Area}(M_n\cap B_N(q_n,\delta_n))}{\pi \delta_n^2}\geq \frac 52.
\end{equation}

If $q\in \Delta$ is NOT a point of self-intersection for $M_\infty$,
by equation~\eqref{density1} we can fix $r>0$ arbitrarily small such that
\[
\frac{\textrm{Area}(M_\infty\cap [B_N(q,2r)-B_N(q,r)])}{\pi r^2}<\frac 72.
\]
However, by equations~\eqref{8/3},~\eqref{genuszero} and~\eqref{positivegenus},
\[
\frac{\mathrm{Area}(M_n\cap [B_N(q_n,2r)- B_N(q_n,r)])}{\pi r^2}\geq 4>\frac 72.
\]
Since the convergence away from the singular points is smooth with multiplicity one, it holds that
\[
\lim_{n\to\infty}\frac{\mathrm{Area}(M_n\cap [B_N(q_n,2r)- B_N(q_n,r)])}{\pi r^2}
= \frac{\textrm{Area}(M_\infty\cap [B_N(q,2r)-B_N(q,r)])}{\pi r^2}.
\]
When $n$ is sufficiently large, this leads to a contradiction. This
finishes the proof that $\Delta$ is contained in the set of points of
self-intersection of $M_\infty$, which finishes the proof of the theorem.

The previous
argument can be used to rule out the occurrence of singular points that are not
catenoid singular points.

\begin{proposition}\label{singnogenus}
Any singular point in $\Delta$ is a catenoid singular point.
\end{proposition}
\begin{proof}
Suppose $q\in\Delta$ is a singular point that is not a
catenoid singular point. By equation~\eqref{density1} and equation~\eqref{density2}
we can fix $r>0$ arbitrarily small such that
\[
\frac{\textrm{Area}(M_\infty\cap [B_N(q,2r)-B_n(q,r)])}{\pi r^2}<\frac{13}2.
\]
However, by equation~\eqref{8/3} and~\eqref{positivegenus},
\[
\frac{\mathrm{Area}(M_n\cap [B_N(q_n,2r)- B_N(q_n,r)])}{\pi r^2}\geq \frac 83 \cdot \frac 52>\frac{13}2,
\]
where $\lim_{n\to\infty}q_n=q$.
Since the convergence away from the singular points is smooth with multiplicity one,
this leads to a contradiction, for $n$ sufficiently large.
\end{proof}


\section{Analysis of the area bound  and the proofs of Theorems~\ref{area} and \ref{seq-compact}} \label{sec:10}

In this section we prove that the area bound only depends on certain properties of the flat 3-torus.
Throughout this section we let $\cT(d,I_0)$ be the space of   flat 3-tori, satisfying
\ben \item $d$ is an upper bound on the
diameter of $N$; \item $I_0>0$ is a lower bound on the injectivity radius of $N$.
\een
Recall that the first property gives that the volume of each flat
3-torus in $\cT(d,I_0)$ is bounded from above by $\frac {\pi d^3}6$.
Note also that we can view $\cT(d,I_0)$ as a compact set of Riemannian metrics on the smooth
manifold $T=\esf^1\times\esf^1\times \esf^1$.
Since the universal covers of each flat 3-torus in
 $\cT(d,I_0)$ are all $\rth$ with the flat metric, we can view these flat 3-tori to be quotients
 of $\rth$ by smoothly varying latices.

We now prove Theorem~\ref{area}. Arguing by contradiction, suppose
that $f_n\colon M_n \to N_n$ is a sequence of closed,
possibly disconnected,
$H$-surfaces where:\ben
\item $N_n\in \cT(d,I_0)$;
\item the area of $M_n$ is greater than $n$;
\item the genus of $M_n$ is at most some fixed $g\in\N$.
\een
Suppose that the flat 3-tori $N_n$ converge to a flat 3-torus $N$,
i.e., the metrics converge to a flat metric on
$T$.  Then we can view the injective mappings $f_n\colon M_n \to N_n$
to correspond to quasi-isometric mappings
into $N$.

There are two cases to consider.  If the injectivity radii of
the surfaces $M_n$ are bounded away from
zero, then the norms of their second fundamental forms are bounded and the
argument in the proof of Claim~\ref{injec} gives a contradiction; more precisely,
the surfaces $M_n$ have uniform regular $\ve$-neighborhoods on their
mean convex sides that,
after replacing by a subsequence, converge smoothly with multiplicity one to
a regular $\ve$-neighborhood on the mean convex side of a smooth strongly
Alexandrov embedded closed surface of genus
at most $g$ and such convergence has multiplicity one.

Suppose now that, after replacing by a subsequence, the injectivity radius of $M_n$ is
less than $1/n$ and that, after choosing a subsequence,  there is a
point $q_1\in N$ and points $p_n\in f_n(M_n)\subset N$
where  the $I_{M_n}(p_n)<1/n$; here we are viewing $f_n(M_n)$ as being
contained in both $N_n$ and the related limit $N$.
Arguing exactly as in the proof of Theorem~\ref{area2}, we can define in a
new subsequence (also labeled as $M_n$) and
a set $\Delta\subset N$ which is the set of singular points of
convergence for $M_n$.  As previously, rescaling arguments
on the scale of the injectivity radius, show that for
any $q\in \Delta$, we can find points $p_n'\in M_n$
of almost minimal injectivity radius such that in small balls
in $N$ centered at the points $p_n'$,
the surfaces $M_n$ have the
appearance  of a complete
properly embedded minimal surfaces $M_\infty$ in
$\rth$ with finite genus at most $g$ or a parking garage structure of $\rth$
with two oppositely handed columns.

As in our previous study, after replacing by a subsequence, we may assume that
at most $g$ points in $\Delta$ can produce a limit minimal surface $M_\infty$ of positive genus.
As before, the only possible limit minimal surface $M_\infty$ of genus zero is the catenoid.
From this point on, all of the arguments that go into the proof of Theorem~\ref{area2} work to show that
the set $\Delta$ is finite and all of these points correspond to the case that the limit surface  $M_\infty$
that forms near them is of catenoid type; as before, these arguments also yield a contradiction
to the assumption that the areas of the originally chosen surfaces $M_n$ is unbounded.
This contradiction completes the proof of Theorem~\ref{area}.

\begin{remark}{\em
The arguments in this section can be applied to prove that  Theorem~\ref{seq-compact} holds in
the more general
setting, namely when the surface $M_n$ lies in a flat 3-torus $N_n$
whose injective radius is at least $I_0>0$ and whose diameter is bounded from above by some $d>0$.
In this case a subsequence of the 3-tori converge to a flat 3-torus $N_\infty$ and a subsequence
of the related surfaces converge
to a strongly Alexandrov embedded surface $M_\infty$ in $N_\infty$.
}\end{remark}

%
%
%

\section{Some images of
genus-3 H-surfaces in flat 3-tori and the dependence of $A(g,a,b,d,I_0)$ on its variables}
One way to construct examples of triply periodic surfaces of non-zero constant mean curvature 1
in $\rth$ is to solve Plateaus problem for a geodesic polygon in the
-sphere $\esf^3=\{|x|=1 \mid x\in \R^4\}$, isometrically map these least area surfaces
into $\rth$ by the Lawson correspondence~\cite{la1},
and then extend them to all of $\rth$ by reflections and translations.
In Figure~\ref{star} we present several images of the fundamental regions of such constant mean curvature
surfaces in a fundamental region of the flat 3-torus $\rth/\Z^3$.
These images were kindly given to us by Karsten  Gro\ss e-Brauckmann.

\begin{figure}[h]
\begin{center}
\includegraphics[width=1.9in]{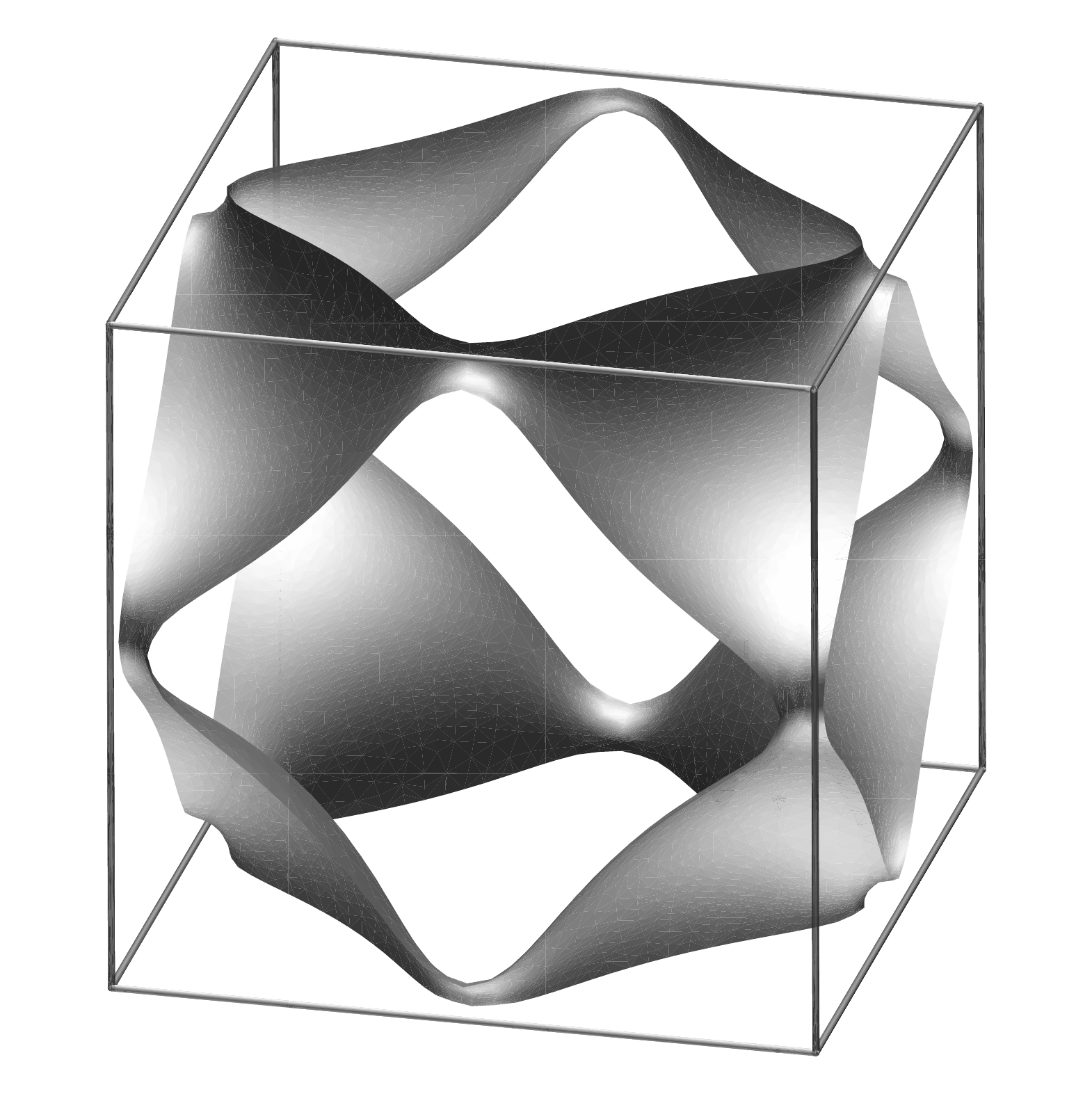}
\includegraphics[width=1.9in]{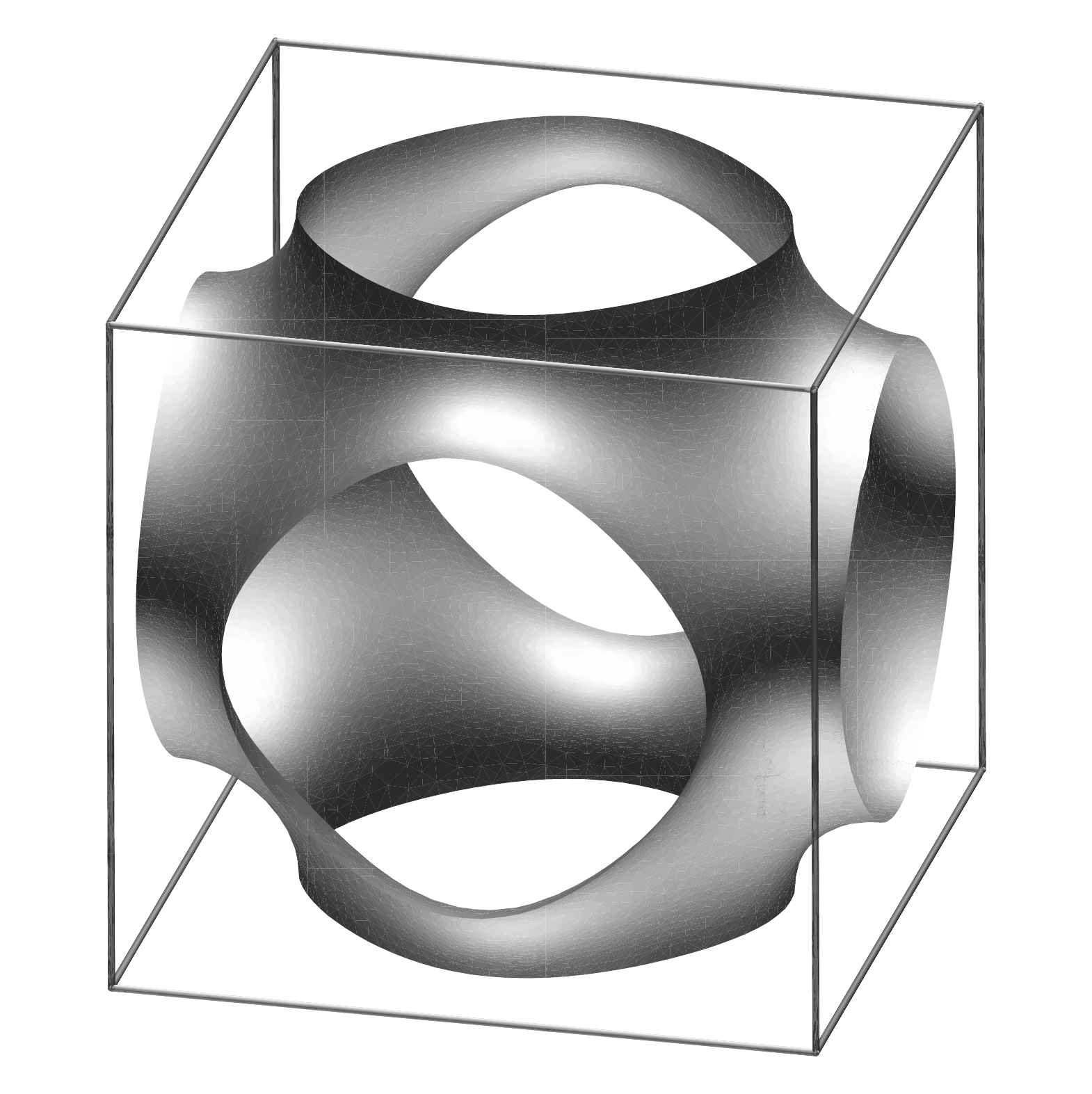}
\\
\includegraphics[width=1.9in]{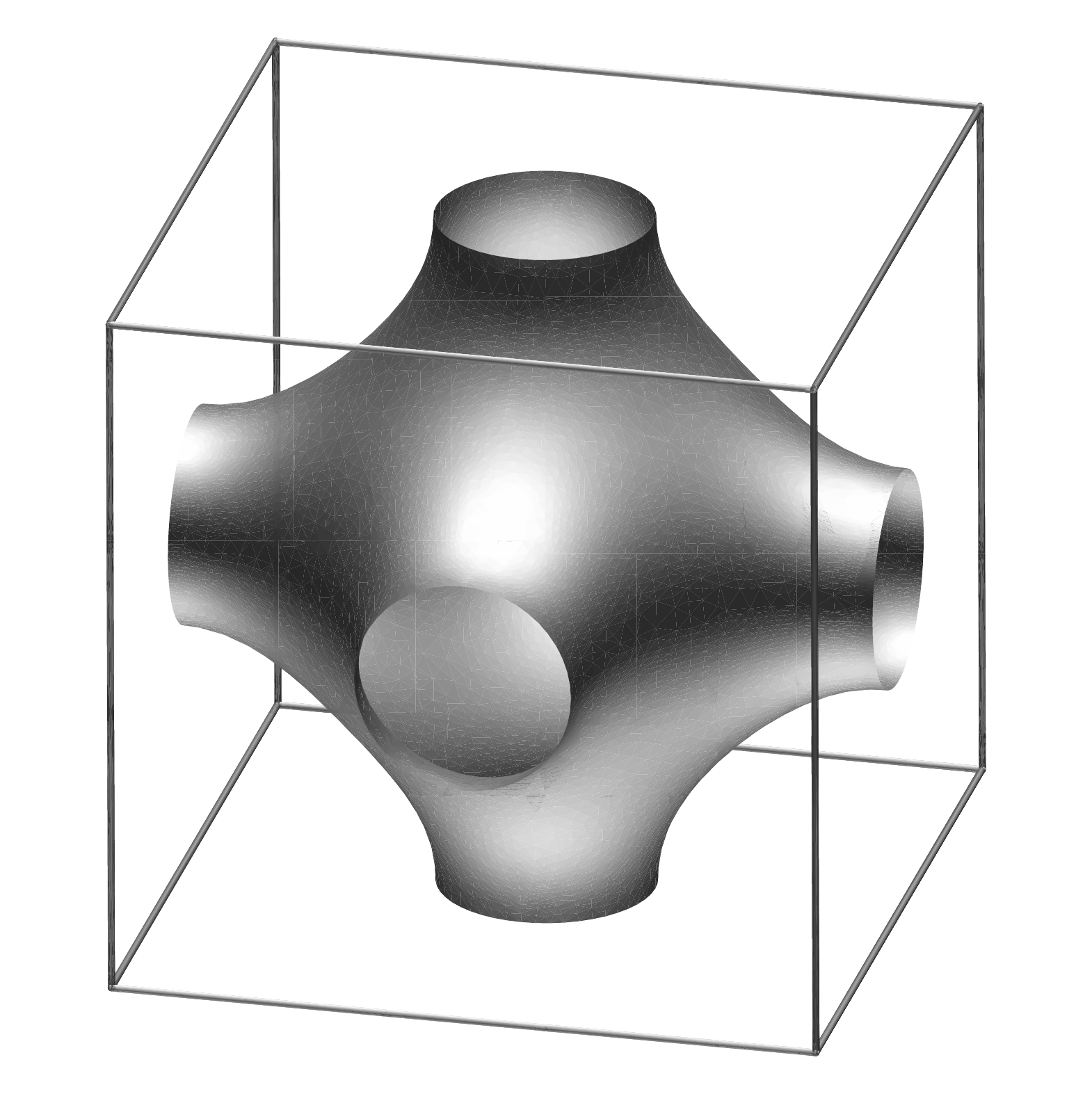}
\includegraphics[width=1.9in]{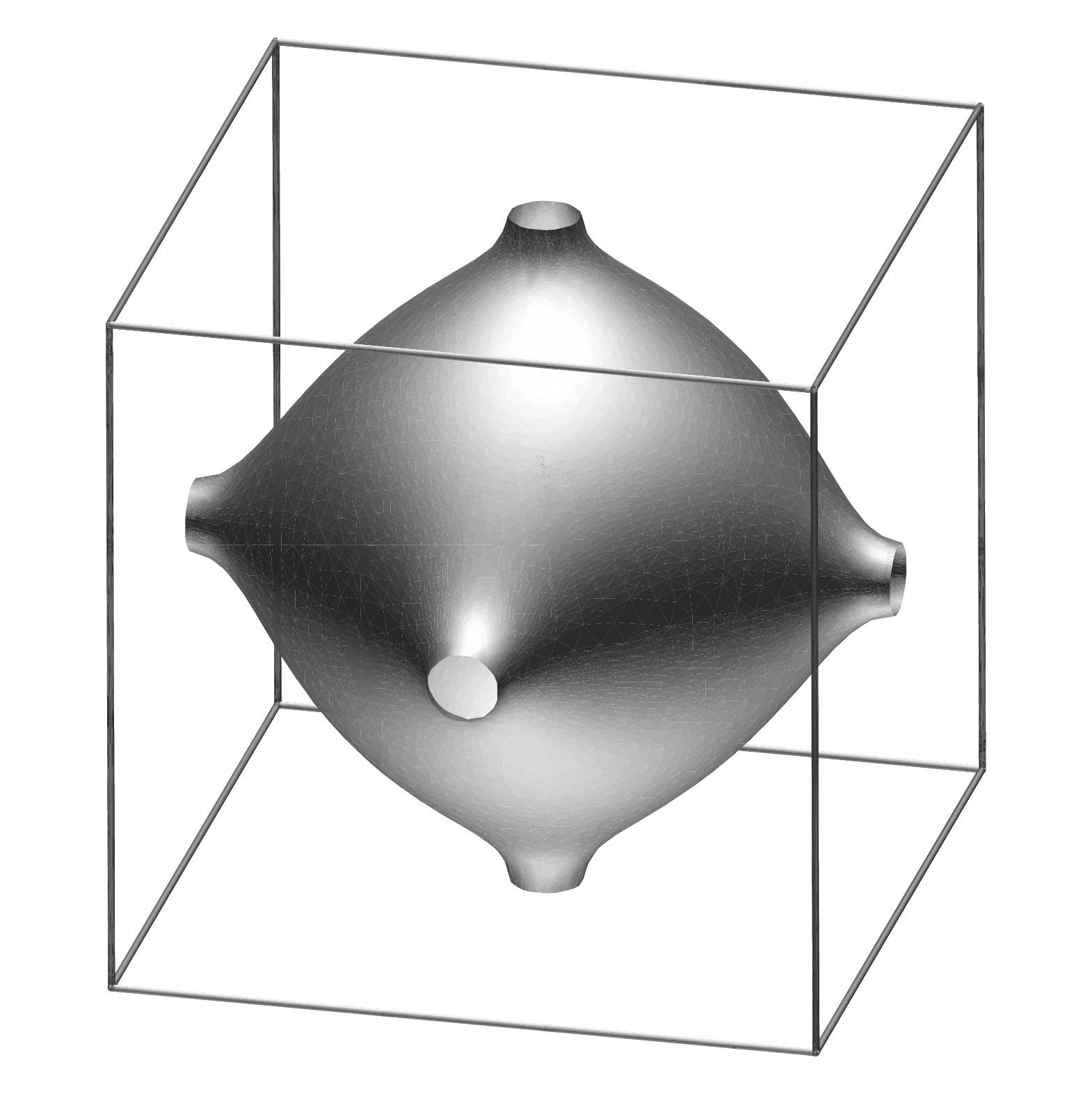}
\end{center}

\caption{Presented above are 4 examples of genus-3 triply periodic constant mean curvature surfaces
in a fundamental region of the 3-torus $\rth/\Z^3$. The two surfaces in the
top row have mean curvature vectors pointing away from the
center of the box. The other two surfaces
have mean curvature vectors pointing towards the center of the box.
These images have been kindly provided by Karsten  Gro\ss e-Brauckmann.}
 \label{star} \end{figure}

\begin{remark} \label{remark}
{\em
By way of examples, it can be shown that the choice of the constant $A(g,a,b,d,I_0)$
in the statement of Theorem~\ref{area}
must depend on the variables $g,a,b,d$, once  $I_0$ is given. We now indicate
without proof what these examples are.
\ben[1.]
\item {\bf  Dependence on an upper bound of $H$:} Every flat 3-torus $T$
admits for any $n$ a collection of
pairwise disjoint geodesic spheres of fixed radius and with total area greater than $n$; the
constant mean curvatures of
these  sphere necessarily goes to infinity as $n$ tends to infinity.
Also note that there exist connected examples that
are flat geodesic cylinders around long closed geodesics in $T$ and that
have arbitrarily large area.
\item {\bf Dependence on a positive lower bound of $H$:}  If $H$ is not
bounded from below, then
the flat 3-torus $T=\rth/\Z^3$ admits  $H_n$-surfaces $M_n$ of genus 3 with
areas greater than $n$ and with $H_n\in (0,1/n)$;
these surfaces can be seen as small deformations of genus-3 minimal
surfaces in $T$ with area greater than
$n$ and geometrically
$M_n$ have the appearance
of being an ``almost totally geodesic'' 2-torus in $T$ and with two attached small ``almost-catenoids,''
where one of these ``almost-catenoids'' is placed at the origin
$\overline{(0,0,0)}\in T$ and the other one is placed near the half-lattice
point $\overline{(1/2,1/2,1/2)}\in T$.
\item {\bf Dependence on an upper bound for the diameter $d$:}
The 3-torus $T_n=\rth/(\Z\times\Z\times n\Z)$
contains intrinsically flat vertical ``cylinders" (flat 2-tori) $C_n$ of "radius 1/3"
and  height $n$ with area $2n\pi /3$.
\item {\bf Dependence on an upper bound for $g$:} In~\cite{mt16}
we construct disconnected
closed $1$-surfaces $M_n$ of area greater than $n$ in some flat 3-torus $T_n$ and such  that
the 3-tori $T_n$ converge to the flat 3-torus  $T=\rth/\Z^3$ as $n$ tends to infinity.
We hope to prove  that the surfaces  $M_n$ can  be chosen to be connected.
\een
}\end{remark}

\section{Outstanding Problems}
 The following outstanding problems are closely related to Theorems~\ref{area} and \ref{seq-compact};
these problems also provided our original motivations for the results in this paper.
It follows from the proofs of Theorems~\ref{area} and \ref{seq-compact} that if $M_n\subset N$ is
a sequence of $H_n$-surfaces satisfying the hypotheses
 in Theorem~\ref{seq-compact}
that converges to the limit surface $M_\infty$ given in its conclusion, then:
{\em Let $q\in N$ be a singular point
of convergence of the $M_n$ to $M_\infty$.  Then for any $\ve>0$ sufficiently small,
there exists an $N_0=N_0(\ve)$ such that for $n\geq N_0$,
$\Sigma_n=\overline{B}_N(q,\ve )\cap M_n$ is a connected compact
surface with two boundary components.}

\begin{conjecture}[Genus-zero Singular Points of Convergence Conjecture] \label{genus-0}
\quad \\
For $\ve>0$ sufficiently small and $n$ sufficiently large, the compact
surface $\Sigma_n =\overline{B}_N(q,\ve )\cap M_n$ is  annulus of total absolute
 Gaussian  curvature $C(\Sigma_n)\in (4\pi-\ve, 4\pi+\ve)$.
\end{conjecture}

The next conjecture is motivated by the compactness
result Theorem~\ref{seq-compact}. In contrast to this conjecture, recall that
Traizet~\cite{tra5}
proved that for any positive integer $g\neq 2$ and $n\in \N$, every flat 3-torus contains
an embedded, connected closed minimal
surface of genus $g$ with area greater than $n$.

\begin{conjecture}[Finiteness Conjecture] \label{finite}
For any $H>0$ and $g\in \N\cup \{0\}$, the moduli space of
non-congruent, connected closed $H$-surfaces of at most genus $g$ in
a fixed  flat 3-torus is finite.
\end{conjecture}

\begin{definition}{\em
A complete $H$-surface in a Riemannian 3-manifold $X$ is said to have {\em locally
finite genus} if for every point $p\in X$, there exists an $\ve_p>0$ such that
the genus of $M\cap B_X(p,\ve_p)$ is bounded. If for some $\ve>0$, the upper bound $U$
on the genus of $M\cap B_X(p,\ve)$ is independent of the point $p$, then we say that
$M$ has {\em $\ve$-uniformly bounded genus with bound $U$}.
}
\end{definition}

\begin{conjecture}[Embedded Calabi-Yau Problem for Locally Finite Genus]
Let $M$ be
a complete $1$-surface in $\rth$.
\ben
\item $M$ is proper in $\rth$ if and only if it has locally bounded genus in $\rth$. Furthermore, this same
properness result holds for
complete, non-planar {\em minimal} surfaces embedded in $\rth$.

\item Given $\ve,U>0$, there exists  $A(\ve,U)>0$ such that if $M$ has $\ve$-uniformly bounded genus
with bound $U$,
then, for all $p\in \rth$,
$$\mathrm{Area}(M\cap \B(p,\ve))\leq A(\ve, U).$$
%
\een
\end{conjecture}

\begin{remark}{\em The area estimates  given in
Theorem~\ref{area} should hold in the following more general context.
Let $N$ be a closed orientable Riemannian 3-manifold.
Given positive numbers $a\leq b$ and $g\in \N\cup \{0\}$, there exists
positive number  $A(g,a,b)$ depending only on $g,a,b$ and $N$
such that the areas of a closed $H$-surfaces $M$ with genus $g$
and $H\in [a,b]$ is less than $A(g,a,b)$, under the assumption that
$M$ is the oriented boundary of a subdomain of $N$.  In particular, if $N$
is a Riemannian homology $3$-sphere, then there is an
area estimate
for connected, closed $H$-surfaces $M$ with fixed finite genus $g$ and $H\in [a,b]$.
This generalization is work in progress in~\cite{mt17}}.
\end{remark}
\vspace{.4cm}

\center{William H. Meeks, III at profmeeks@gmail.com}\\
Mathematics Department, University of Massachusetts, Amherst, MA
01003
\center{Giuseppe Tinaglia at giuseppe.tinaglia@kcl.ac.uk}\\
Department of Mathematics, King's College London,  London, WC2R 2LS, U.K.

\bibliographystyle{plain}
\bibliography{bill}
\end{document}